\documentclass[11pt]{amsart}

\usepackage[utf8]{inputenc} 
\usepackage[T1]{fontenc} 
\usepackage[english]{babel} 
\usepackage{mathrsfs} 
\usepackage{amssymb}

\theoremstyle{plain} 
\newtheorem{thm}{Theorem} 
\newtheorem{lem}[thm]{Lemma} 
\newtheorem{cor}[thm]{Corollary}

\theoremstyle{remark} 
\newtheorem*{remark}{Remark} 
\newtheorem*{question}{Question}

\DeclareMathOperator{\mre}{Re} 
 
\DeclareMathOperator{\rank}{rank} 
\DeclareMathOperator{\codim}{codim}

\newcommand{\DD}{\mathbb D}

\begin{document} 
\title[Inequalities for Bergman spaces and Hankel forms]{Contractive 
inequalities for Bergman spaces and multiplicative Hankel forms}
\date{\today} 

\author[Bayart]{Fr\'{e}d\'{e}ric Bayart} \address{Université Clermont Auvergne, 
CNRS, LMBP, F-63000 Clermont–Ferrand, France} 
\email{bayart@math.univ-bpclermont.fr}

\author[Brevig]{Ole Fredrik Brevig} \address{Department of Mathematical 
Sciences, Norwegian University of Science and Technology (NTNU), NO-7491 
Trondheim, Norway} \email{ole.brevig@math.ntnu.no}

\author[Haimi]{Antti Haimi} \address{Faculty of Mathematics, University of 
Vienna, Oskar-Morgenstern-Platz 1, 1090 Wien, Austria} 
\email{antti.haimi@univie.ac.at}

\author[Ortega-Cerd\`{a}]{Joaquim Ortega-Cerd\`{a}} \address{Department de 
Matem\`{a}tiques i Inform\`{a}tica, Universitat de Barcelona \& Barcelona 
Graduate school in mathematics, 
Gran Via 585, 08007 Barcelona, Spain} \email{jortega@ub.edu}

\author[Perfekt]{Karl-Mikael Perfekt} \address{Department of Mathematics, The 
University of Tennessee, Knoxville, TN 37996, USA} \email{kperfekt@utk.edu}

\thanks{The second named author is supported by Grant 227768 of the Research Council of Norway. The third named author is supported by Lise Meitner grant of Austrian Science Fund (FWF). The fourth named author is supported by the MTM2014-51834-P grant by the Ministerio
de Econom\'{\i}a y Competitividad, and by the Generalitat de Catalunya (project 2014 SGR 289).}

\begin{abstract}
We consider sharp inequalities for Bergman spaces of the unit disc, 
establishing analogues of the inequality in Carleman's proof of the 
isoperimetric inequality and of Weissler's inequality for dilations. By 
contractivity and a standard tensorization procedure, the unit disc inequalities 
yield corresponding inequalities for the Bergman spaces of Dirichlet series. We 
use these results to study weighted multiplicative Hankel forms associated with 
the Bergman spaces of Dirichlet series, reproducing most of the known results on 
multiplicative Hankel forms associated with the Hardy spaces of Dirichlet 
series. In addition, we find a direct relationship between the two type of forms 
which does not exist in lower dimensions. Finally, we produce some 
counter-examples concerning Carleson measures on the infinite polydisc.
\end{abstract}


\maketitle

\section{Introduction} \label{sec:intro} 
Hardy spaces of the countably infinite polydisc, $H^p(\DD^\infty)$, have in 
recent years received considerable interest and study, emerging from the 
foundational papers \cite{CG86,HLS97}. Partly, the attraction is motivated by 
the subject's link with Dirichlet series, realized by identifying each complex 
variable with a prime Dirichlet monomial, $z_j=p_j^{-s}$ (see~\cite{Bayart02}). 
Hardy spaces of Dirichlet series, $\mathscr{H}^p$, are defined by requiring 
this identification to induce an isometric, multiplicative isomorphism. The  
connection to Dirichlet series gives rise to a rich interplay between operator 
theory and analytic number theory --- we refer the interested reader to the 
survey \cite{Qu15} or the monograph \cite{QQ13} as a starting point.

One aspect of the theory is the study of multiplicative Hankel forms on 
$\ell^2\times\ell^2$. A sequence $\varrho = (\varrho_1,\varrho_2,\ldots)$ 
generates a multiplicative Hankel form by the formula
\begin{equation}
	\label{eq:mhank} \varrho(a,b) = \sum_{m=1}^\infty \sum_{n=1}^\infty a_m 
b_n \varrho_{mn}, 
\end{equation}
defined at least for finitely supported sequences $a$ and $b$. Helson 
\cite{Helson05} observed that multiplicative Hankel forms are naturally 
realized as (small) Hankel operators on $H^2(\mathbb{D}^\infty)$, and went on to 
ask whether every symbol $\rho$ which generates a bounded multiplicative Hankel 
form on $\ell^2 \times \ell^2$ also induces a bounded linear functional on the 
Hardy space $H^1(\mathbb{D}^\infty)$. In other words, he asked whether there is 
an analogue of Nehari's theorem \cite{Nehari57} in this context. 

Helson's question inspired several papers 
\cite{BP15,BPSSV,Helson06,Helson10,OCS12, PP16}. Following the program outlined 
in \cite{Helson10}, it was established in \cite{OCS12} that there are bounded 
Hankel forms that do not extend to bounded functionals on 
$H^1(\mathbb{D}^\infty)$. In the positive direction, it was proved in 
\cite{Helson06} that if the Hankel form \eqref{eq:mhank} instead satisfies the 
stronger property of being Hilbert--Schmidt, then its symbol does extend to a 
bounded functional on $H^1(\mathbb{D}^\infty)$. Briefly summarizing the most 
recent development, the result of \cite{OCS12} was generalized in \cite{BP15}, 
in \cite{BPSSV} an analogue of the classical Hilbert matrix was introduced and 
studied, and in \cite{PP16} the boundedness of the Hankel form \eqref{eq:mhank} 
was characterized in terms of Carleson measures in the special case that the 
form is positive semi-definite.

Very recently, a study of Bergman spaces of Dirichlet series $\mathscr A^p$ 
begun in \cite{BL15}. In analogy with the Hardy spaces of Dirichlet series, $ 
\mathscr A^p$ is constructed from the corresponding Bergman space, 
$A^p(\DD^\infty)$. New difficulties appear in trying to put this theory on 
equal footing with its Hardy space counterpart. One of them is the lack of 
contractive inequalities for Bergman spaces in the unit disc. In the Hardy space 
of the unit disc there is a comparative abundance of such inequalities, each 
immediately implying a corresponding inequality for $\mathscr{H}^p$. For 
example, the result of \cite{Helson06} on Hilbert--Schmidt Hankel forms relies 
essentially on the classical Carleman inequality,
\[\|f\|_{A^2(\mathbb{D})} \leq \|f\|_{H^1(\mathbb{D})}.\]
A second example is furnished by Weissler's inequality: defining for $0 < r 
\leq 1$ the map $P_r \colon H^p(\mathbb{D}) \to H^q(\mathbb{D})$, by $P_rf(w) = 
f(rw)$, then $P_r$ is contractive if and only if $r \leq \sqrt{p/q} \leq 1.$ 
Since both of these inequalities are contractive, they carry on to the infinite 
polydisc by tensorization (see~Section \ref{sec:iter}), thus yielding results 
for $\mathscr H^p$.

We derive analogues of the mentioned inequalities for Bergman spaces of the 
unit disc in Section~\ref{sec:ineq}. Our proofs involve certain variants of the 
Sobolev inequalities from \cite{Bandy09} and \cite{Beckner93}. Then, in 
Section~\ref{sec:iter}, we follow the by now standard tensorization scheme to 
deduce the corresponding contractive inequalities for the Bergman spaces of 
Dirichlet series.

Section \ref{sec:hankel} is devoted to the weighted multiplicative Hankel forms 
related to the Bergman space, defined by the formula
\begin{equation}
	\label{eq:mhankd} \varrho_d(a,b) = \sum_{m=1}^\infty \sum_{n=1}^\infty 
a_m b_n \frac{\varrho_{mn}}{d(mn)}, \qquad a, b \in \ell^2_d.
\end{equation}
In \eqref{eq:mhankd}, $d(k)$ denotes the number of divisors of the integer $k$, and $\ell^2_d$ 
denotes the corresponding weighted Hilbert space. Note that the divisor function $d(k)$ 
counts the number of times $\varrho_k$ appears in \eqref{eq:mhankd}. In the 
same way that the forms \eqref{eq:mhank} are realized as Hankel operators on 
the Hardy space $H^2(\mathbb{D}^\infty)$, the weighted forms \eqref{eq:mhankd} 
are naturally realized as (small) Hankel operators on the Bergman space of the 
infinite polydisc, $A^2(\mathbb{D}^\infty)$. Equipped with the inequalities 
from Sections \ref{sec:ineq} and \ref{sec:iter} we successfully obtain the 
Bergman space counterparts of results from 
\cite{BPSSV,Helson06,Helson10,OCS12}. 

In Section~\ref{sec:hankel} we will also point out a surprising property of 
multiplicative Hankel forms. We first observe that $A^2(\mathbb{D}^\infty)$ may 
be naturally isometrically embedded in the Hardy space 
$H^2(\mathbb{D}^\infty)$, since the same is true for $A^2(\mathbb{D})$ with 
respect to $H^2(\mathbb{D}^2)$. Then, we notice that this embedding lifts to the 
level of Hankel forms, giving us natural map taking weighted Hankel forms 
\eqref{eq:mhankd} to Hankel forms \eqref{eq:mhank}. The striking aspect is that 
this map preserves the singular numbers of the Hankel form, in particular 
preserving both the uniform and the Hilbert--Schmidt norm.

Finally, in Section \ref{sec:carleson} we come back to harmonic analysis on the 
Hardy spaces $H^p(\DD^\infty)$. We produce two counter-examples for Carleson 
measures, again pointing out phenomena that do not exist in finite dimension. 

\subsection*{Notation} We will use the notation $f(x)\lesssim g(x)$ if there is 
some constant $C>0$ such that $|f(x)|\leq C|g(x)|$ for all (appropriate) $x$. 
If $f(x)\lesssim g(x)$ and $g(x) \lesssim f(x)$, we write $f(x) \simeq g(x)$. 
As above, $(p_j)_{j\geq1}$ will denote the increasing sequence of prime numbers.

\section{Inequalities of Carleman and Weissler for Bergman spaces} 
\label{sec:ineq} 
\subsection{Preliminaries} Let $\alpha>1$ and $0< p<\infty$, and define the 
Bergman space $A^p_\alpha(\mathbb{D})$ as the space of analytic functions $f$ 
in the unit disc 
\[\mathbb{D} = \{z\,:\, |z|<1\}\]
that are finite with respect to the norm
\[\|f\|_{A^p_\alpha(\mathbb{D})} = \left(\int_{\mathbb{D}} |f(w)|^p \, 
(\alpha-1)(1-|w|^2)^{\alpha-2}\,dm(w)\right)^\frac{1}{p}.\]
Here $m$ denotes the Lebesgue area measure, normalized so that 
$m(\mathbb{D})=1$. It will be convenient to let $dm_\alpha(w) = 
(\alpha-1)(1-|w|)^{\alpha-2}\,dm(w)$ for $\alpha > 1$, and to let $m_1$ denote 
the normalized Lebesgue measure on the torus
\[\mathbb{T} = \{z\,:\, |z|=1\}.\]
The Hardy space $H^p(\mathbb{D})$ is defined as closure of analytic polynomials 
with respect to the norm
\[\|f\|_{H^p(\mathbb{D})} = \left(\int_{\mathbb{T}}|f(w)|^p 
\,dm_1(w)\right)^\frac{1}{p}.\]
The Hardy space $H^p(\mathbb{D})$ is the limit of $A^p_\alpha(\mathbb{D})$ as 
$\alpha \to 1^+$, in the sense that
\[\lim_{\alpha\to 1^+} \|f\|_{A^p_\alpha(\mathbb{D})} = 
\|f\|_{H^p(\mathbb{D})}\]
for every analytic polynomial $f$. We therefore let $A^p_1(\mathbb{D}) = 
H^p(\mathbb{D})$. Our main interest is in the distinguished case $\alpha=2$, 
when $m_\alpha = m$ is simply the normalized Lebesgue measure. Therefore we 
also let $A^p(\mathbb{D})=A^p_2(\mathbb{D})$. We will only require some basic 
properties of $A^p_\alpha(\mathbb{D})$ in what follows, and refer generally to 
the monographs \cite{DS04,HKZ00}.

Let $c_\alpha(j)$ denote the coefficients of the binomial series 
\begin{equation}
	\label{eq:binomseries} \frac{1}{(1-w)^\alpha} = \sum_{j=0}^\infty 
c_\alpha(j) w^j, \qquad c_\alpha(j) = \binom{j+\alpha-1}{j}. 
\end{equation}
It is evident from \eqref{eq:binomseries} that 
\begin{equation}
	\label{eq:sumconv} \sum_{j+k=l} c_\alpha(j)c_\beta(k) = 
c_{\alpha+\beta}(l).
\end{equation}
If $\alpha$ is an integer, then $c_\alpha(j)$ denotes the number of ways to 
write $j$ as a sum of $\alpha$ non-negative integers. Furthermore, if $f(w) = 
\sum_{j\geq0} a_j w^j$, then 
\begin{equation}
	\label{eq:A2anorm} \|f\|_{A^2_\alpha(\mathbb{D})} = 
\left(\sum_{j=0}^\infty \frac{|a_j|^2}{c_\alpha(j)}\right)^\frac{1}{2}. 
\end{equation}
Functions $f$ in $A^p_\alpha(\mathbb{D})$ satisfy for $w \in \mathbb{D}$ the 
sharp pointwise estimate 
\begin{equation}
	\label{eq:pest} |f(w)| \leq 
\frac{1}{(1-|w|^2)^{\alpha/p}}\|f\|_{A^p_\alpha(\mathbb{D})}. 
\end{equation}

For the sake of completeness, we will state and prove the results in this 
section for as general $\alpha>1$ as we are able, even though we will only make 
use of the results for $\alpha=2$ in the following sections.

\subsection{Contractive inclusions of Bergman spaces} It is well-known that, if 
$0<p\leq q$ and $\alpha,\beta\geq 1$, then $A^p_\alpha(\DD)$ embeds 
continuously into $A^q_\beta(\DD)$ if and only if $q/\beta\leq p/\alpha$ (see 
e.g. \cite[Exercise 2.27]{Zhu05}). By tensorization, this statement extends to 
the Bergman spaces on the polydiscs of finite dimension. However, in order for 
such embeddings to exist on the infinite polydisc, it is necessary that the 
inclusion map in one variable is contractive.

The first result of the type we are looking for was given by Carleman 
\cite{Car21}. For $f\in H^1(\DD)$ it holds that
\begin{equation}
	\label{eq:carlemanH1} \|f\|_{A^2(\mathbb{D})} = \|f\|_{A^2_2(\DD)}\leq 
\|f\|_{A^1_1(\DD)}=\|f\|_{H^1(\DD)}. 
\end{equation}

A modern and natural way to prove \eqref{eq:carlemanH1} can be found in 
\cite{Vukotic03}. First, it is easy to verify that
$$\|gh\|_{A^2(\mathbb{D})} \leq 
\|g\|_{H^2(\mathbb{D})}\|h\|_{H^2(\mathbb{D})},$$
for example by computing by coefficients. If $f$ is a non-vanishing function of 
$H^1(\DD)$, writing $f=gh$ with $g=h=f^{1/2}$ now leads to 
\eqref{eq:carlemanH1}. For a general function $f\in H^1(\DD)$, we first factor 
out the zeroes through a Blaschke product. This is possible by what seems to be 
a coincidence: multiplication by a Blaschke product decreases the norm on the 
left hand side of \eqref{eq:carlemanH1} but preserves the norm on the right 
hand side.

The ability to factor out zeroes and take roots implies that Carleman's 
inequality \eqref{eq:carlemanH1} holds for arbitrary $0<p<\infty$, 
$$\|f\|_{A^{2p}(\mathbb{D})}\leq\|f\|_{H^p(\mathbb{D})}.$$
 In \cite{Burbea87}, Burbea generalized Carleman's inequality, showing that for 
every $0<p<\infty$ and every non-negative integer $n$, it holds that
\begin{equation}
	\label{eq:burbea} \|f\|_{A^{p(1+n)}_{1+n}(\mathbb{D})} \leq 
\|f\|_{H^p(\mathbb{D})}. 
\end{equation}
Let 
\[\alpha_0 = \frac{1+\sqrt{17}}{4}=1.280776\ldots\]
We offer the following extension of Carleman's inequality. 
\begin{thm}
	\label{thm:generalcarleman} Let $\alpha\geq\alpha_0$ and $0<p<\infty$. 
For every $f\in A^p_\alpha(\DD)$, 
	\[\|f\|_{A^{p(\alpha+1)/\alpha}_{\alpha+1}(\DD)}\leq 
\|f\|_{A^p_\alpha(\DD)}.\]
	Moreover, if $\alpha>\alpha_0$, we have equality if and only if there 
exists constants $C\in\mathbb C$ and $\xi\in \DD$ such that 
	\[f(w)=\frac{C}{(1-\bar\xi w)^{2\alpha/p}}.\] 
\end{thm}
 Let us give two corollaries. The first is mainly decorative, but it 
illustrates that \eqref{eq:burbea} gets weaker as $n$ increases. 
\begin{cor}
	\label{cor:chain} Let $f \in H^1(\mathbb{D})=A^1_1(\mathbb{D})$. Then
	\[\|f\|_{A^1_1(\DD)} \geq \|f\|_{A^2_2(\DD)} \geq \|f\|_{A^3_3(\DD)} 
\geq \|f\|_{A^4_4(\DD)} \geq \cdots \]
\end{cor}
We also have the following corollary, which will be important in the next 
section. 
\begin{cor}
	\label{cor:HLin} Let $p = 2/(1+n/2)$ for a non-negative integer $n$ and 
suppose that $f(w) = \sum_{j\geq0} a_j w^j$ is in $A^p(\mathbb{D})$. Then
	\[\|f\|_{A^2_{n+2}(\mathbb{D})} = \left(\sum_{j=0}^\infty 
\frac{|a_j|^2}{c_{n+2}(j)}\right)^\frac{1}{2} \leq \|f\|_{A^p(\mathbb{D})}.\]
\end{cor}
\begin{proof}
	This follows from $n$ successive applications of 
Theorem~\ref{thm:generalcarleman}, starting from $p=2/(1+n/2)$ and $\alpha=2$.
\end{proof}

We now begin the proof of Theorem \ref{thm:generalcarleman}. A version of it 
was announced in \cite{Bandy09}\footnote{Theorem~3.2 in \cite{Bandy09} is stated 
for $kq>2$, but there seems to be a mistake in the proof of uniqueness on 
p.~1083. The argument in its entirety seems to apply only when $kq>3$.}, 
following a scheme designed in \cite{Bodmann04}. Observe also that an analogous 
result in the Fock space was proved by Carlen \cite{Carlen91} using a 
logarithmic Sobolev inequality. We follow the general strategy of 
\cite{Bandy09,Bodmann04}, replacing \cite[Sec.~5]{Bandy09} with a result from 
\cite{MS08}. We include many additional details in an attempt to make the scheme 
used in \cite{Bandy09,Bodmann04,Carlen91} available to a wider audience.

We shall use two structures on the disk, the Euclidean and the hyperbolic. The 
usual gradient and Laplacian of $u$ will be denoted by $\nabla u $ and $\Delta 
u$, while the hyperbolic gradient and the hyperbolic Laplacian are denoted by 
$\nabla_{\operatorname{H}}\, u$ and $\Delta_{\operatorname{H}}\, u$. They are 
connected by the following formulas: 
\[\nabla_{\operatorname{H}}\, u(w)=\left(\frac{1-|w|^2}{2}\right)\nabla 
u(w)\qquad \text{and} \qquad\Delta_{\operatorname{H}}\, u(w) 
=\left(\frac{1-|w|^2}{2}\right)^2\Delta u(w).\]
We shall also use the M\"obius invariant measure
\[d\mu(w)=\frac{dm(w)}{(1-|w|^2)^2}.\] 
We begin with an integral identity (essentially \cite[Thm.~3.1]{Bandy09}). An 
analogous result was proven for the Fock space in \cite{Carlen91}, and a 
similar result also appears in \cite{Bodmann04}. 
\begin{lem}
	\label{lem:carlenid} Let $p>0$ and $\beta>1/2$. For an analytic 
function $f$ in $\overline{\mathbb{D}}$, set $u(w)=|f(w)|^p (1-|w|^2)^{\beta}$. 
Then
\[\int_{\mathbb{D}} 
|\nabla_{\operatorname{H}}\, u(w)|^2 d\mu(w) = \frac \beta 2\int_{\mathbb{D}} 
|u(w)|^2 d\mu(w).\]
\end{lem}
\begin{proof}
	Integrating by parts gives 
	\begin{equation}
		\label{eq:carlenid1} \int_{\DD}|\nabla_{\operatorname{H}}\, 
u|^2d\mu=\frac 14\int_{\mathbb{D}} |\nabla u|^2 dm= 
-\frac{1}4\int_{\mathbb{D}}u \Delta u dm. 
	\end{equation}
	It follows from the assumption $\beta > 1/2$ that boundary terms do not 
appear here. We compute the Laplacian now. At any point where $f$ does not 
vanish, we can write 
	\begin{eqnarray*}
		\frac{
		\partial u}{\bar
		\partial w}=\frac p2 
|f|^{p-2}f\overline{f'}(1-|w|^2)^\beta-\beta w |f|^p (1-|w|^2)^{\beta-1},
	\end{eqnarray*}
	so that 
	\begin{eqnarray*}
		\frac{
		\partial^2 u}{
		\partial w\bar
		\partial 
w}&=&\frac{p^2}4|f'|^2|f|^{p-2}(1-|w|^2)^{\beta}-\beta\frac p2 |f|^{p-2} 
f\overline{f'}\bar w(1-|w|^2)^{\beta-1}\\
		&&\qquad-\,\beta |f|^p (1-|w|^2)^{\beta-1}-\frac{\beta p}{2} 
|f|^{p-2}f'\bar f(1-|w|^2)^{\beta-1}\\
		&&\qquad\qquad+\,\beta(\beta -1)|w|^2|f|^p(1-|w|^2)^{\beta-2}. 
	\end{eqnarray*}
	We see that 
	\begin{eqnarray*}
		-u\Delta u&=&-p^2 |f'|^2 |f|^{2p-2} (1-|w|^2)^{2\beta}+2\beta p 
|f|^{2p-2} f\overline{f'} \bar w (1-|w|^2)^{2\beta-1}\\
		&&\qquad+\,4\beta |f|^{2p}(1-|w|^2)^{2\beta-2}+2\beta p 
|f|^{2p-2} f'\bar f w(1-|w|^2)^{2\beta-1}\\
		&&\qquad\qquad-\,4\beta^2 |w|^2 |f|^{2p}(1-|w|^2)^{2\beta-2}. 
	\end{eqnarray*}
	Coming back to the expression of $\partial u /\bar\partial w$, we find 
that
	\[-\frac 14u\Delta u = \beta \frac{u^2}{(1-|w|^2)^2}-\left|\frac{
		\partial u}{\bar
		\partial w}\right|^2=\beta 
\frac{u^2}{(1-|w|^2)^2}-\frac{|\nabla_{\operatorname{H}}\, u|^2}{(1-|w|^2)^2}. 
\]
	Integrating with respect to $dm$ and using \eqref{eq:carlenid1} gives 
the result. 
\end{proof}

\begin{proof}
	[Proof of Theorem~\ref{thm:generalcarleman}] We set 
$q=p(\alpha+1)/\alpha$, $A=(\alpha-2)/(\alpha-1)$ and $B=1/(\alpha-1)$, so that 
$A+B=1$. We want to find the infimum of 
	\[(\alpha-1)\int_{\DD} |f(w)|^p (1-|w|^2)^\alpha d\mu(w)\]
	under the constraint
	\[\alpha\int_{\DD} |f(w)|^q (1-|w|^2)^{\alpha+1}d\mu(w)=1.\]
	Equivalently, using Lemma \ref{lem:carlenid} with 
	\begin{equation} \label{eq:uf}
		u(w)=|f(w)|^{p/2}(1-|w|^2)^{\alpha/2},
	\end{equation} we want to find the infimum of 
	\begin{equation}
		\label{eq:minimizer1} 
A\int_{\DD}|u(w)|^2d\mu(w)+\frac{4B}{\alpha}\int_{\DD}|\nabla_{\operatorname{H}}
\, u(w)|^2d\mu(w) 
	\end{equation}
	under the constraint 
	\begin{equation}
		\label{eq:minimizer2} \alpha\int_{\DD}|u(w)|^{2q/p}d\mu(w)=1. 
	\end{equation}
	We now solve the latter minimization problem for real-valued $u$ 
belonging to the Sobolev space $W^{1,2}(\DD)$, i.e. functions $u$ such that
	$$\int_{\DD}|\nabla_{\operatorname{H}}\, u(w)|^2d\mu(w) < \infty.$$
	 By the well-known inequality for the bottom of the spectrum of the 
Laplace--Beltrami operator (see e.g.~\cite{MS08}) we know that for any $u\in 
W^{1,2}(\DD)$,
	\[\int_{\DD}|u(w)|^2 d\mu(w)\leq 
4\int_{\DD}|\nabla_{\operatorname{H}}\, u(w)|^2d\mu(w).\]
	Hence
	\[N(u)=\left(A\int_{\DD} |u(w)|^2d\mu(w)+\frac{4B}\alpha 
\int_{\DD}|\nabla_{\operatorname{H}}\, u(w)|^2d\mu(w)\right)^{1/2}\]
	 is a norm on $W^{1,2}(\DD)$ equivalent to the usual norm, since 
$A>-B/\alpha$. By the Rellich--Kondrakov theorem \cite[Ch.~11]{leoni2009first}, which 
asserts that the inclusion map from $W^{1,2}(\DD)$ into $L^s(\DD,d\mu)$ is 
compact for any finite $s$, the problem of finding the infimum of 
\eqref{eq:minimizer1} for $u\in W^{1,2}(\DD)$ satisfying \eqref{eq:minimizer2} 
is well-posed. Moreover, this also ensures that minimizers do exist. Indeed, 
let us take any sequence $(u_n)$ realizing the infimum. This sequence is bounded 
in the reflexive space $W^{1,2}(\DD)$, so we may assume that it converges weakly 
to some $u\in W^{1,2}(\DD)$. Then $(u_n)$ converges to $u$ in 
$L^{2q/p}(\DD,d\mu)$ so that $\|u\|^{2q/p}_{L^{2q/p}}=1/\alpha$
	 whereas $N(u)\leq \liminf_n N(u_n)$.
	
	Next we compute the Euler--Lagrange equation corresponding to the 
constrained variational problem given by \eqref{eq:minimizer1} and 
\eqref{eq:minimizer2}. By standard arguments, we find that any local minimum of the problem 
is a weak solution of
	\begin{equation}
		\label{eq:mineq} Au-\frac{4B}\alpha \Delta_{\operatorname{H}}\, 
u=\lambda u^{\frac{2q}p-1} 
	\end{equation}
	for some $\lambda \in \mathbb{R}$. By Lemma~\ref{lem:reg} below, there 
are minimizers that are actually $C^2(\mathbb{D})$. Multiplying by $u$ and integrating 
with respect to $\mu$, we find from \eqref{eq:carlenid1} that $\lambda>0$. We 
now rescale \eqref{eq:mineq} by setting $u=\kappa v$ with
	\[\kappa^{2q/p-2}=\frac{4B}{\alpha\lambda}.\]
	Then $v\in W^{1,2}(\DD)\cap C^2(\mathbb{D})$ satisfies 
	\begin{equation}
		\label{eq:minimizer3} \Delta_{\operatorname{H}}\, 
v-\frac{(\alpha-2)\alpha}4v+v^{\frac{2q}p-1}=0. 
	\end{equation}
	We now investigate \eqref{eq:mineq} for our candidate 
solution $u_0(w)=(1-|w|^2)^{\alpha/2}$. Since
	\[\Delta_{\operatorname{H}}\, u_0(w)=-\frac\alpha 
2(1-|w|^2)^{\alpha/2}\left(1-\frac\alpha2|w|^2\right)\]
	we have that
	\[Au_0-\frac{4B}\alpha \Delta_{\operatorname{H}} 
u_0=\frac\alpha{\alpha-1}(1-|w|^2)^{\frac\alpha 2+1}=\lambda_0 
u_0^{\frac{2q}{p}-1},\]
	where $\lambda_0=\alpha/(\alpha-1)$. Hence, if we let $u_0=\kappa_0 
v_0$ with 
	\[\kappa_0^{2q/p-2}=\frac{4B}{\alpha\lambda_0},\]
	then $v_0\in W^{1,2}(\DD)$ is a solution of \eqref{eq:minimizer3}. 
However, by \cite[Thm.~1.3]{MS08} we know that the solution of 
\eqref{eq:minimizer3} is unique up to a M\"obius transformation, as long as 
$$\frac{\alpha(2-\alpha)}4<\frac{4q}{p\left(\frac {2q}p+2\right)^2}.$$ 
	Replacing $q/p$ by its value, we find that this inequality is satisfied 
if and only if $\alpha>\alpha_0$. Both the Euler--Lagrange equation and our 
constraint problem are invariant under M\"obius transformations, so we have 
found all minimizers. Coming back to analytic functions via \eqref{eq:uf}, we 
have shown that we have equality if and only if there exists $\xi\in\DD$ and 
$\widetilde{C}\in\mathbb R$ such that
	\[|f(w)|^{p/2} = \widetilde{C}\,\left|1-\left|\frac{\xi-w}{1-\bar \xi 
w}\right|^2\right|^{\alpha/2}\big|1-|w|^2\big|^{-\alpha/2} = \widetilde{C}\, 
\frac{\left(1-|\xi|^2\right)^{\alpha/2}}{\left|1-\bar\xi w\right|^{\alpha}}.\]
	This shows that $f$ has to be a multiple of $(1-\bar\xi 
w)^{-2\alpha/p}$ for some $\xi\in\mathbb{D}$. Finally, the assertion of the 
theorem for $\alpha=\alpha_0$ is obtained by taking the limit as $\alpha \to 
\alpha_0^+$.
\end{proof}
The following is the regularity result that was used in the proof of the previous theorem. 
\begin{lem} \label{lem:reg}
There are minimizers of the variational constrained variational problem given 
by \eqref{eq:minimizer1} and \eqref{eq:minimizer2} that are $C^2$ smooth in $\mathbb{D}$.
\end{lem}
\begin{proof}
Let $u$ be a minimizer. Then it is weak solution of the Euler-Lagrange equation 
\eqref{eq:mineq}. We also know that 
$u \in L^{2q/p}(\mathbb{D}, d\mu)$. Since the radial rearrangement decreases 
the Dirichlet norm (by the Polya--Szeg\"o inequality 
\cite[Thm.~16.17]{leoni2009first}) there is a minimizer $u$ that is 
positive, radially symmetric and decreasing. Therefore $F(u)$ is bounded in the
unit disk, where
$$F(u):=  \frac{\alpha}{4B} \big( Au - \lambda u^{\frac{2q}{p}-1} \big)$$ 
Consider any solution $v$ to the Poisson equation:
\[
\Delta v(z) = \frac{F(u(z))}{(1-|z|^2)^2} ,
\]
then $u -v$ satisfies $\Delta(u-v) = 0$ weakly. Therefore $u = v + h $ where $h$ 
is an harmonic function.
One explicit solution to the Poisson equation is given by
$$
v(z) =\int_{\mathbb D}K(z,w)\frac{F(u(w))}{(1-|w|^2)^2} \,dm(w)
$$
where
$$
K(z,w)=\frac{1}{2\pi}\left\{\log
\left\vert\frac{w-z}{1-\overline{w}z}\right\vert^2+\frac{
(1-|w|^2)(1-|z|^2)}
{|1-\overline{w}z|^2}+\vert
z\vert^2\left(\frac{1-\vert w\vert^2}{\vert
1-\overline{w}z\vert}\right)^2\right\}.
$$
It was shown in \cite{andersson1985solution} that $K(z,w)$ satisfies the estimate 
$$
|K(z,w)|\lesssim
\frac{(1-\vert w\vert^2)^2}{|1-\overline{w}z|^2}\left(1+\log
\left|\frac{1-\overline{w}z}{w-z}\right|\right),\qquad z,w\in\mathbb
D.
$$
The difference between $u$ and $v$ is harmonic, thus the regularity of $u$ 
follows from the regularity of $v$.
\end{proof}
\begin{remark}
	The constants $A$ and $B$, with $A + B =1$, were chosen in the proof so 
that $u(w)=(1-|w|^2)^{\alpha/2}$ would be a solution of the Euler--Lagrange 
equation for some $\lambda\in\mathbb R$. This is only possible if 
$\beta=\alpha+1$, and thus explains why this relationship is imposed in the 
statement of Theorem \ref{thm:generalcarleman}. The condition 
$\alpha\geq\alpha_0$ comes from \cite[Thm.~1.3]{MS08}, but we do not know if it 
is necessary for the uniqueness of \eqref{eq:minimizer3}.
\end{remark}
\begin{question}
	For any $0<p\leq q$ and $\alpha,\beta\geq 1$ such that $q/\beta\leq 
p/\alpha$, does the contractive inequality $$\|f\|_{A^q_\beta(\DD)}\leq 
\|f\|_{A^p_\alpha(\DD)}$$
	hold? By Carleman's inequality and Theorem \ref{thm:generalcarleman}, 
this is true when $\beta=\alpha+n$ for some integer $n$, and either $\alpha=1$ 
or $\alpha\geq\alpha_0$. We remark that it is easy to show, for example by 
computing with coefficients, that
	\[\|f\|_{A^4_{2\alpha}(\mathbb{D})}\leq \|f\|_{A^2_\alpha(\mathbb{D})}\]
	holds for every $\alpha\geq1$.
\end{question}

\subsection{Hypercontractivity of the Poisson kernel} For $r\in[0,1]$, let 
$P_r$ denote the operator defined on analytic functions in $\mathbb{D}$ by $P_rf(w) = f(rw)$. 
Clearly, if $r<1$ it follows from \eqref{eq:pest} that $P_r$ maps any 
$A^p_\alpha(\mathbb{D})$ into every $A^q_\beta(\mathbb{D})$. We are interested 
in knowing when this map is contractive.
\begin{thm}
	\label{thm:weissler} Let $0<p\leq q<\infty$ and let $\alpha = (n+1)/2$ 
for some $n\in\mathbb{N}$. Then $P_r$ is a contraction from 
$A^p_\alpha(\mathbb{D})$ to $A^q_\alpha(\mathbb{D})$ if and only if $r \leq 
\sqrt{p/q}$. 
\end{thm}

Weissler \cite{Weissler80} proved Theorem~\ref{thm:weissler} when $\alpha=1$. 
The case $\alpha=3/2$ is also known, see \cite[Remark~5.14]{Gross02} or 
\cite{Janson83}, but it appears that these are the only two previously 
demonstrated cases. To prove Theorem~\ref{thm:weissler} we will use a classical 
argument of complex analysis to transfer results from Hardy spaces to Bergman 
spaces in smaller dimensions. This will be accomplished through the following 
lemma. 
\begin{lem}[\cite{Rudin80}, Sec.~1.4.4]
	\label{lem:slices} Let $\mathbb{S}^n$ denote the real unit sphere of 
dimension $n\geq 1$, and let $\sigma_n$ denote its normalized surface measure. 
Extend the function $h \colon \mathbb{D} \to \mathbb{C}$ to $\mathbb{S}^n$ by $\widetilde{h}(x)=h(x_1+ix_2)$ for $x = 
(x_1,x_2,\ldots,x_{n+1})\in\mathbb{S}^n$. Then
	\[\int_{\mathbb{S}^n} \widetilde{h}(x)\, d\sigma_n(x) = 
\int_{\mathbb{D}} h(w) dm_{(n+1)/2}(w).\]
\end{lem}

We can now demonstrate how Theorem~\ref{thm:weissler} follows from a result of 
Beckner \cite{Beckner93} concerning the unit sphere.
\begin{proof}
	[Proof of Theorem~\ref{thm:weissler}] Let $\mathcal{P}_r$ denote the 
Poisson kernel on $\mathbb{S}^n$, defined by
	\[\mathcal{P}_r(\xi,\eta)=\frac{1-r^2}{|r\xi-\eta|^{n+1}}, \qquad 
\xi,\eta \in \mathbb{S}^n.\]
	For a function $g$ on $\mathbb{S}^n$, let
	
\[(\mathcal{P}_rg)(\xi)=\int_{\mathbb{S}^n}\mathcal{P}_r(\xi,
\eta)g(\eta)d\sigma_n(\eta).\]
	It is proved in \cite{Beckner93} that $\mathcal{P}_r$ defines a 
contraction from $L^s(\mathbb{S}^n)$ to $L^t(\mathbb{S}^n)$, $1\leq s\leq 
t<\infty$, if and only if $r\leq \sqrt{(s-1)/(t-1)}$.
	
	Let us now start with $0<p \leq q<\infty$ and $r<\sqrt{p/q}$. Let $m$ 
be a large number such that $mp>1$ and such that
	\[r \leq \sqrt{\frac{mp-1}{mq-1}}.\]
	Given an analytic polynomial $f$, we define $g$ on $\mathbb{S}^n$ by
	\[g(x_1,x_2,\ldots,x_{n+1}) = |f(x_1+ix_2)|^{1/m}.\]
	Since $f$ is analytic, it follows that $g$ is subharmonic and hence for 
any $(x_1,\dots,x_{n+1})\in \mathbb{S}^n$ we get that
	\[g(rx_1,\dots,rx_{n+1})\leq \mathcal{P}_rg(x_1,\dots,x_{n+1}).\]
	Using Beckner's result with $s = mp$ and $t = mq$ we get that
	
\[\left(\int_{\mathbb{S}^n}g(rx_1,\dots,rx_{n+1})^{mq}d\sigma_n(x)\right)^{1/q} 
\leq 
\left(\int_{\mathbb{S}^n}g(x_1,\dots,x_{n+1})^{mp}d\sigma_n(x)\right)^{1/p}.\]
	By Lemma~{\ref{lem:slices}}, this is the same as
	\[\left(\int_{\mathbb{D}}|f(rw)|^q dm_{(n+1)/2}(w)\right)^\frac{1}{q} 
\leq \left(\int_{\mathbb{D}} |f(w)|^p dm_{(n+1)/2}(w)\right)^\frac{1}{p}.\]
	It follows that the condition $r \leq \sqrt{p/q}$ is sufficient (by a 
limiting argument in the endpoint case $r = \sqrt{p/q}$). Conversely, for fixed 
$r>0$ and small $\varepsilon>0$ we have that
	\[\left(\int_{\mathbb{D}} |1+ \varepsilon r w|^q\,dm_\alpha(w) 
\right)^\frac{1}{q} = 1 + \frac{q r^2}{4\alpha}\,\varepsilon^2 + 
O(\varepsilon^4).\]
	Letting $\varepsilon \to 0$ shows that $qr^2 \leq p$ is also necessary, 
for any value of $\alpha \geq 1$.
\end{proof}
\begin{remark}
	As in the previous subsection, we conjecture that 
Theorem~\ref{thm:weissler} is true for all values of $\alpha\geq 1$. Several 
other positive results can be deduced from Theorem \ref{thm:generalcarleman}. 
For instance, if $\alpha\geq\alpha_0$, then
\[\|P_r f\|_{A_\alpha^2(\DD)}\leq 
\|f\|_{A^{2\alpha/(\alpha+1)}_{\alpha}(\DD)},\]
for every analytic polynomial $f$, if and only if $r^2\leq (\alpha+1)/\alpha$. 
In fact, it follows from Theorem~\ref{thm:generalcarleman} that
\[\|f\|_{A^2_{\alpha+1}(\DD)}\leq \|f\|_{A_\alpha^{2\alpha/(\alpha+1)}(\DD)}.\]
Computing the norms as in \eqref{eq:A2anorm}, we have that
\[\|P_r f\|_{A^2_{\alpha}(\DD)}\leq \|f\|_{A^2_{\alpha+1}(\DD)}\]
if and only if, for any $k\geq 1$,
\[r^{2k}\leq \frac{c_{\alpha+1}(k)}{c_\alpha(k)}=\frac{\alpha+k}{\alpha}.\] 
\end{remark}

\section{Inequalities on the polydisc and in the half-plane} \label{sec:iter} 
For $\alpha>1$, consider the following product measure on $\mathbb{D}^\infty$,
\[\mathbf{m}_\alpha(z) = m_\alpha(z_1)\times m_\alpha(z_2) \times m_\alpha(z_3) 
\times \cdots,\]
and for $0<p<\infty$ the corresponding Lebesgue space 
$L^p_\alpha(\mathbb{D}^\infty)$. We define the Bergman spaces of the infinite 
polydisc, denoted $A^p_\alpha(\mathbb{D}^\infty)$, as the closure in 
$L^p_\alpha(\mathbb{D}^\infty)$ of the space of analytic polynomials in an 
arbitrary number of variables. The Hardy spaces $H^p(\mathbb{D}^\infty)$ are 
defined as the closure of analytic polynomials with respect to the norm given 
by the product $m_1 \times m_1 \times \cdots$ on $\mathbb{T}^\infty$, so that
\[\|f\|_{H^p(\mathbb{D}^\infty)}^p = \int_{\mathbb{T}^\infty} |f(z)|^p 
\,d\textbf{m}_1(z).\] 
As before, $H^p(\mathbb{D}^\infty)$ is the limit as $\alpha \to 1^+$ of 
$A^p_\alpha(\mathbb{D}^\infty)$, in the sense that
\[\lim_{\alpha\to1^+} \|f\|_{A^p_\alpha(\mathbb{D}^\infty)} = 
\|f\|_{H^p(\mathbb{D}^\infty)}\]
for every analytic polynomial $f$. We distinguish the case $\alpha=2$ by 
writing $A^p(\mathbb{D}^\infty)=A^p_2(\mathbb{D}^\infty)$. Applying the point 
estimate \eqref{eq:pest} repeatedly we find that if $f$ is a polynomial in 
$A^p_\alpha(\mathbb{D}^\infty)$, then 
\begin{equation}
	\label{eq:pestinf} |f(z)| \leq \left(\prod_{j=1}^\infty 
\frac{1}{1-|z_j|^2}\right)^{\alpha/p} \|f\|_{A^p_\alpha(\mathbb{D}^\infty)}, 
\end{equation}
which implies that elements of $A^p_\alpha(\mathbb{D}^\infty)$ are analytic 
functions on $\mathbb{D}^\infty \cap \ell^2$. Every $f$ in 
$A^p_\alpha(\mathbb{D}^\infty)$ has a power series expansion convergent in 
$\mathbb{D}^\infty \cap \ell^2$,
\begin{equation}
	\label{eq:formalexp} f(z) = \sum_{\kappa \in \mathbb{N}^\infty_0} 
a_\kappa z^\kappa, 
\end{equation}
where $\mathbb{N}^\infty_0$ denotes the set of all finite non-negative 
multi-indices. 

Finally, when $p=2$ we can compute the norm explicitly. Suppose that $f$ is of 
the form \eqref{eq:formalexp}. Then 
\begin{equation}
	\label{eq:A2infnorm} \|f\|_{A^2_\alpha(\mathbb{D}^\infty)} = 
\left(\sum_{\kappa \in \mathbb{N}^\infty_0} 
\frac{|a_\kappa|^2}{c_\alpha(\kappa)}\right)^\frac{1}{2}, \qquad\text{ where 
}\qquad c_\alpha(\kappa) = \prod_{j=1}^\infty c_\alpha(\kappa_j). 
\end{equation}
Note that the final product contains only a finite number of factors not equal 
to $1$, since $\kappa$ is a finite multi-index.

The contractive inequalities of Section~\ref{sec:ineq} can now be extended to 
$\mathbb{D}^\infty$ using Minkowski's inequality in the following formulation: 
if $X$ and $Y$ are measure spaces, $g$ a measurable function on $X \times Y$, 
and $p\geq1$, then
\[\left(\int_X\left(\int_Y |g(x,y)|\,dy\right)^p\,dx\right)^\frac{1}{p} \leq 
\int_Y \left(\int_X |g(x,y)|^p\,dx\right)^\frac{1}{p}\,dy.\]
It is sufficient to prove the contractive results on the finite polydiscs 
$\mathbb{D}^d$, $d<\infty$, as this allows us to conclude by the density of 
analytic polynomials. This is done by iteratively applying the one dimensional 
result to each of the variables, and applying Minkowski's inequality in each 
step.
This procedure has been repeated many times (for instance in 
\cite{Bayart02,BHS15,Helson06} or in \cite[Sec.~6.5.3]{QQ13}) and we do not 
include the details here. 

In particular, Corollary~\ref{cor:HLin} for $n=2$ yields the next result on the 
polydisc. Helson \cite{Helson06} proved the corresponding result for the Hardy 
spaces $H^p(\mathbb{D}^\infty)$,
which he used to study Hilbert--Schmidt multiplicative Hankel forms. We shall 
carry out the analogous study for weighted multiplicative Hankel forms 
associated with the Bergman space in the next section.
\begin{lem}
	\label{lem:helson} 
$\|f\|_{A^2_4(\mathbb{D}^\infty)}\leq\|f\|_{A^1(\mathbb{D}^\infty)}$. 
\end{lem}

Let $\mathbf{r} = (r_1,r_2,\ldots)$ with $r_j\in[0,1]$ and define 
$P_{\mathbf{r}} f(z) = f(r_1z_1,r_2z_2,\ldots)$. Following \cite{Bayart02} and 
using Theorem~\ref{thm:weissler} (with $\alpha=2$), we get the next result. 
\begin{lem}
	\label{lem:weisslerpoly} Let $0<p\leq q < \infty$. The map 
$P_{\mathbf{r}}$ is a contraction from $A^p(\mathbb{D}^\infty)$ to 
$A^q(\mathbb{D}^\infty)$ if and only if $r_j \leq \sqrt{p/q}$. Moreover, 
$P_{\mathbf{r}}$ is bounded from $A^p(\mathbb{D}^\infty)$ to 
$A^q(\mathbb{D}^\infty)$ as soon as $r_j \leq \sqrt{p/q}$ for all but a finite 
set of $j$s. 
\end{lem}

When working with multiplicative Hankel forms and Dirichlet series, it is often 
convenient to recast the expansion \eqref{eq:formalexp} in multiplicative 
notation. Each integer $n\geq1$ can be written in a unique way as a product of 
prime numbers,
\[n = \prod_{j=1}^\infty p_j^{\kappa_j}.\]
This factorization associates $n$ uniquely to the finite non-negative 
multi-index $\kappa(n)$. Setting $a_n = a_{\kappa(n)}$, we rewrite 
\eqref{eq:formalexp} as 
\begin{equation}
	\label{eq:multexp} f(z) = \sum_{n=1}^\infty a_n z^{\kappa(n)}. 
\end{equation}

For $\alpha\geq1$ we define the general divisor function $d_\alpha(n)$ as the 
coefficients of the Dirichlet series given by $\zeta^\alpha$, where 
$\zeta(s)=\sum_{n\geq1} n^{-s}$ is the Riemann zeta function. Using the Euler 
product of the Riemann zeta function, say for $\mre(s)>1$, we find that 
\begin{equation}
	\label{eq:zetapow} \zeta(s)^\alpha = \left(\prod_{j=1}^\infty 
\frac{1}{1-p_j^{-s}}\right)^\alpha = \prod_{j=1}^\infty\left(\sum_{k=0}^\infty 
c_\alpha(k)p_j^{-ks}\right) =\sum_{n=1}^\infty d_\alpha(n) n^{-s}.
\end{equation}
It follows that $c_\alpha(\kappa(n)) = d_\alpha(n)$. In multiplicative 
notation, we restate \eqref{eq:A2infnorm} as 
\[\left\|\sum_{n=1}^\infty a_n 
z^{\kappa(n)}\right\|_{A^2_\alpha(\mathbb{D}^\infty)} = \left(\sum_{n=1}^\infty 
\frac{|a_n|^2}{d_\alpha(n)}\right)^\frac{1}{2}.\]
When $\alpha\geq1$ is an integer, it is clear that $d_\alpha(n)$ denotes the 
number of ways to write $n$ as a product of $\alpha$ non-negative integers. In 
particular, $d_2$ is the usual divisor function $d$. It also follows from 
\eqref{eq:zetapow} that 
\begin{equation}
	\label{eq:multconv} \sum_{mn=l} d_\alpha(m)d_\beta(n) = 
d_{\alpha\beta}(l),
\end{equation}
in analogy with \eqref{eq:sumconv}. 

The Bohr lift of a Dirichlet series $f(s) = \sum_{n\geq 1} a_n n^{-s}$ is the 
power series defined by
\[\mathscr{B}f(z) = \sum_{n=1}^\infty a_n z^{\kappa(n)},\]
realizing the identification $z_j = p_j^{-s}$. The Bergman space of Dirichlet 
series $\mathscr A^p$ is defined as the completion of Dirichlet polynomials in 
the norm $$\|f\|_{\mathscr{A}^p} = \|\mathscr{B}f\|_{A^p(\mathbb{D}^\infty)}.$$
Inequality \eqref{eq:pestinf} implies that $\mathscr{A}^p$ is a space of 
analytic functions in the half-plane $\mathbb{C}_{1/2}$, and that $f$ in 
$\mathscr{A}^p$ enjoys the sharp pointwise estimate
\begin{equation} \label{eq:Appointwise}
|f(s)| \leq \zeta(2\mre{s})^{2/p} \|f\|_{\mathscr{A}^p}.
\end{equation}
Let $\mathscr{T}$ denote the conformal map of $\mathbb{D}$ to 
$\mathbb{C}_{1/2}$ given by
\[\mathscr{T}(z) = \frac{1}{2} + \frac{1-z}{1+z}.\]
The conformally invariant Bergman space of the half-plane $\mathbb{C}_{1/2}$, 
denoted $A^p_{\alpha,\operatorname{i}}(\mathbb{C}_{1/2})$, is the space of 
analytic functions $f$ in $\mathbb{C}_{1/2}$ with the property that $f \circ 
\mathscr{T} \in A^p_{\alpha}(\mathbb{D})$. A computation shows that
\[\|f\|_{A^p_{\alpha,\operatorname{i}}(\mathbb{C}_{1/2})}^p = 
\int_{\mathbb{C}_{1/2}} |f(s)|^p\,(\alpha-1) 
\left(\mre(s)-\frac{1}{2}\right)^{\alpha-2}\,\frac{4^{\alpha-1}}{|s+1/2|^{
2\alpha}}\,dm(s).\]
By Lemma~\ref{lem:helson} we have the following version of Carleman's 
inequality for Dirichlet series in the half-plane. 
\begin{thm}
	\label{thm:carlemanhalf} Suppose that $f(s) = \sum_{n\geq1} a_n n^{-s}$ 
is in $\mathscr{A}^1$. Then 
	\begin{equation}
		\label{eq:curlycarleman} \left(\sum_{n=1}^\infty 
\frac{|a_n|^2}{d_4(n)}\right)^\frac{1}{2} \leq \|f\|_{\mathscr{A}^1} 
	\end{equation}
	Moreover, there is a constant $C\geq1$ such that 
$\|f\|_{A^2_{4,\operatorname{i}}(\mathbb{C}_{1/2})} \leq 
C\|f\|_{\mathscr{A}^1}$. 
\end{thm}
\begin{proof}
	The inequality \eqref{eq:curlycarleman} is Lemma~\ref{lem:helson} in 
multiplicative notation. The second statement follows from the first and 
Example~2 in \cite{Olsen11}. 
\end{proof}

For $\varepsilon>0$, define the translation operator $T_\varepsilon$ by 
$T_\varepsilon f(s) = f(s+\varepsilon)$. Here is a sharp and general version of 
\cite[Prop.~9]{BL15}, which we interpret as Weissler's inequality for Dirichlet 
series in the half-plane. The corresponding result for $\mathscr{H}^p$ can be 
found in \cite{Bayart02}.
\begin{thm}
	\label{thm:weisslerhalf} Let $0<p\leq q < \infty$. The operator 
$T_\varepsilon \colon \mathscr{A}^p \to \mathscr{A}^q$ is bounded for every 
$\varepsilon>0$, and contractive if and only if $2^{-\varepsilon} \leq 
\sqrt{p/q}$. 
\end{thm}
\begin{proof}
	This follows from Lemma~\ref{lem:weisslerpoly}, using the fact that 
$T_\varepsilon$ corresponds to $P_\mathbf{r}$ with $r_j = p_j^{-\varepsilon}$. 
\end{proof}

We end this section by demonstrating that Lemma~\ref{lem:weisslerpoly} also 
implies a weak generalization of Theorem~\ref{thm:carlemanhalf} to more general 
exponents. In the Hardy space context, it was proven in \cite{BHS15} that if $f(s) = \sum_{n\geq1} a_n n^{-s}$ and $0<p\leq 2$, 
then
\[\left(\sum_{n=1}^\infty |a_n|^2 
\frac{|\mu(n)|}{d_{2/p}(n)}\right)^\frac{1}{2} \leq \|f\|_{\mathscr{H}^p}.\]
The M\"{o}bius factor $|\mu(n)|$ is $1$ if $n$ is square-free and $0$ if not. 
From \eqref{eq:burbea}, it follows that this factor may actually be replaced by 
$1$ if $p=2/(1+n)$ for some non-negative integer $n$. We have the following 
extension to Bergman spaces in mind. 
\begin{thm}
	\label{thm:AHLineq} Let $0<p\leq 2$ and suppose that $f(s) = 
\sum_{n\geq1} a_n n^{-s}$ is in $\mathscr{A}^p$. Then
	\[\left(\sum_{n=1}^\infty |a_n|^2 
\frac{|\mu(n)|}{d_{4/p}(n)}\right)^\frac{1}{2} \leq \|f\|_{\mathscr{A}^p}.\]
	If $p = 2/(1+n/2)$ for some non-negative integer $n$, then
	\[\left(\sum_{n=1}^\infty |a_n|^2 
\frac{1}{d_{4/p}(n)}\right)^\frac{1}{2} \leq \|f\|_{\mathscr{A}^p}.\]
\end{thm}
\begin{proof}
	Let $\Omega(n)$ denote the number of prime factors of $n$ (counting 
multiplicity). Using Lemma~\ref{lem:weisslerpoly} with $r_j = \sqrt{p/2}$, we 
have that 
	\begin{align*}
		\left\|\sum_{n=1}^\infty a_n n^{-s}\right\|_{\mathscr{A}^p} 
&\geq \left\|\sum_{n=1}^\infty a_n \left(\frac{p}{2}\right)^{\Omega(n)/2} 
n^{-s}\right\|_{\mathscr{A}^2} = \left(\sum_{n=1}^\infty |a_n|^2 
\frac{1}{(2/p)^{\Omega(n)} d(n)}\right)^\frac{1}{2} \\
		&\geq \left(\sum_{n=1}^\infty |a_n|^2 
\frac{|\mu(n)|}{(2/p)^{\Omega(n)} d(n)}\right)^\frac{1}{2} = 
\left(\sum_{n=1}^\infty |a_n|^2 \frac{|\mu(n)|}{d_{4/p}(n)}\right)^\frac{1}{2}. 
	\end{align*}
	In the final equality we used that $d_\alpha(n) = \alpha^{\Omega(n)}$ 
when $n$ is square-free. When $p = 2/(1+n/2)$ for a non-negative integer $n$, 
tensorizing  Corollary~\ref{cor:HLin} (by appealing to Minkowski's inequality) 
yields that the M\"{o}bius factor is actually unnecessary; see 
Lemma~\ref{lem:helson} and Theorem~\ref{thm:carlemanhalf}.
\end{proof}
\begin{remark}
	Considering the square-free terms only of a Dirichlet series is in many 
cases sufficient to obtain sharp results, see for example \cite{BHS15}. Often, 
the reason for this is related to the fact that the square-free zeta function 
has the same behaviour as the zeta function $\zeta(s)$ near $s=1$, since
	\[\sum_{n=1}^\infty |\mu(n)| n^{-s} = \prod_{j=1}^\infty (1+p_j^{-s}) = 
\prod_{j=1}^\infty \frac{1-p_j^{-2s}}{1-p_j^{-s}} = 
\frac{\zeta(s)}{\zeta(2s)}.\]
\end{remark}

\section{Multiplicative Hankel forms} \label{sec:hankel} The multiplicative 
Hankel form \eqref{eq:mhankd} is said to be bounded if there is a constant 
$C<\infty$ such that 
\begin{equation}
	\label{eq:bmhankd} |\varrho(a,b)| = 
\left|\sum_{m=1}^\infty\sum_{n=1}^\infty a_m b_n \frac{\varrho_{mn}}{d(mn)} 
\right| \leq C \left(\sum_{m=1}^\infty 
\frac{|a_m|^2}{d(m)}\right)^\frac{1}{2}\left(\sum_{n=1}^\infty 
\frac{|b_n|^2}{d(n)}\right)^\frac{1}{2}. 
\end{equation}
The smallest such constant is the norm of $\varrho$. The symbol of the form 
$\varrho$ is the Dirichlet series $\varphi(s) =\sum_{n\geq1} 
\overline{\varrho_n} n^{-s}$. If $f$ and $g$ are Dirichlet series with 
coefficient sequences $a$ and $b$, respectively, then \eqref{eq:bmhankd} can be 
rewritten as $|H_\varphi(fg)|\leq C 
\|f\|_{\mathscr{A}^2}\|g\|_{\mathscr{A}^2}$, where we define
\[H_\varphi(fg) = \langle fg, \varphi \rangle_{\mathscr{A}^2} = 
\sum_{l=1}^\infty \left(\sum_{mn=l} a_m b_n\right)\,\frac{\varrho_l}{d(l)} = 
\sum_{m=1}^\infty \sum_{n=1}^\infty a_m b_n \frac{\varrho_{mn}}{d(mn)}.\]
Hence, the multiplicative Hankel form is bounded if and only if $H_\varphi$ is 
a bounded form on $\mathscr{A}^2 \times \mathscr{A}^2$.

We begin with the following example, giving the Bergman space analogue of the 
multiplicative Hilbert matrix studied in \cite{BPSSV}. Let $\mathscr{A}^2_0$ 
denote the subspace of $\mathscr{A}^2$ consisting of Dirichlet series $f(s) = 
\sum_{n\geq1} a_n n^{-s}$ such that $a_1 = f(+\infty) = 0$. As in \cite{BPSSV}, 
it is natural to work with Dirichlet series without constant term for 
convergence reasons. We consider the form 
\begin{equation}
	\label{eq:hilberttype} H(fg) = \int_{1/2}^\infty 
f(\sigma)g(\sigma)\left(\sigma-\frac{1}{2}\right)\,d\sigma, \qquad f,g \in 
\mathscr{A}^2_0. 
\end{equation}
\begin{thm}
	\label{thm:hilberttype} The bilinear form \eqref{eq:hilberttype} is a 
multiplicative Hankel form with symbol
	\[\varphi(s) = \int_{1/2}^\infty \left(\zeta(s+\sigma)^2 -1 
\right)\left(\sigma-\frac{1}{2}\right)\,d\sigma = \sum_{n=2}^\infty 
\frac{d(n)}{\sqrt{n}(\log{n})^2}n^{-s}.\]
	The form $H_\varphi$ is bounded, but not compact, on 
$\mathscr{A}^2_0\times \mathscr{A}^2_0$. 
\end{thm}
\begin{proof}
	To see that $\varphi$ is the symbol, one can either compute $H(fg)$ at 
the level of coefficients or use that $\zeta(s+\overline{w})^2-1$ is the 
reproducing kernel of $\mathscr{A}^2_0$. To see that $H$ is bounded, we first 
use the Cauchy--Schwarz inequality,
	\[|H(fg)| \leq \left(\int_{1/2}^\infty |f(\sigma)|^2\, 
\left(\sigma-\frac{1}{2}\right)d\sigma\right)^\frac{1}{2}\left(\int_{1/2}
^\infty |g(\sigma)|^2\, 
\left(\sigma-\frac{1}{2}\right)d\sigma\right)^\frac{1}{2}.\]
	By symmetry, we only need to consider one of the factors. We split the 
integral at $\sigma=1$.
	\[\int_{1/2}^\infty |f(\sigma)|^2\, 
\left(\sigma-\frac{1}{2}\right)d\sigma = 
\left(\int_{1/2}^1+\int_1^\infty\right) 
|f(\sigma)|^2\,\left(\sigma-\frac{1}{2}\right) d\sigma.\]
	The first integral is bounded by a constant multiple of 
$\|f\|_{\mathscr{A}^2}^2$, as follows from \cite[Thm.~3 and 
Example~4]{Olsen11}. For the second integral, we have by the pointwise estimate 
\eqref{eq:Appointwise} that
	\[|f(\sigma)|^2 \leq \|f\|_{\mathscr{A}^2}^2 \left(\sum_{n=2}^\infty 
d(n) n^{-2\sigma}\right) \leq (2+o(1))4^{-\sigma} \|f\|_{\mathscr{A}^2}^2,\]
	where we in the final inequality used that $\sigma\geq 1$. To show that 
$H_\varphi$ is not compact, let $k_\varepsilon(s)$ denote the normalized 
reproducing kernel of $\mathscr{A}^2_0$ at 
	the point $1/2 + \varepsilon/2$,
	
$$k_{\varepsilon}(s)=\frac{\zeta^2(s+1/2+\varepsilon/2)-1}{\sqrt{
\zeta^2(1+\varepsilon)-1}}.$$ The functions $k_\varepsilon$ converge weakly to 
$0$ as $\varepsilon \to 0$, since they converge to $0$ on every compact subset 
of $\mathbb{C}_{1/2}$. By the fact that
	\[\zeta(s) = \frac{1}{s-1} + O(1)\]
	for $\mre(s)>1$ close to $1$, we get for, say $1/2 < \sigma < 1$, that
	\[k_\varepsilon(\sigma) = \frac{(\sigma+1/2+\varepsilon/2-1)^{-2} + 
O(1)}{(1+\varepsilon-1)^{-1}+O(1)} = \varepsilon 
\left(\frac{1}{(\sigma-1/2+\varepsilon/2)^2}+O(1)\right).\]
	Setting $f = g = k_\varepsilon$, we find that
	\[H(fg) = \varepsilon^2 \left(\int_{1/2}^1 
\left(\frac{1}{(\sigma-1/2+\varepsilon/2)^4}+O(1)\right) 
\left(\sigma-\frac{1}{2}\right)\,d\sigma + O(1)\right) \gtrsim 1, \]
	showing that $H$ is not compact. 
\end{proof}

Since the Bohr lift is multiplicative, it holds that
\[\langle fg, \varphi \rangle_{\mathscr{A}^2} = \langle 
\mathscr{B}f\mathscr{B}g,\mathscr{B}\varphi \rangle_{A^2(\mathbb{D}^\infty)}.\]
For the remainder of this section we will work in the polydisc, and we 
therefore tacitly identify the Dirichlet series $f$ with its Bohr lift 
$\mathscr{B}f$. Hence, we consider symbols of the form
\[\varphi(z) = \sum_{n=1}^\infty \overline{\varrho_n} z^{\kappa(n)},\]
and define $H_\varphi(fg) = \langle fg, \varphi 
\rangle_{A^2(\mathbb{D}^\infty)}$, for $f,g \in A^2(\mathbb{D}^\infty)$. 

If $\varphi$ defines a bounded functional on $A^1(\mathbb{D}^\infty)$, then it 
follows from the Cauchy--Schwarz inequality that
\[|H_\varphi(fg)| = |\langle fg, \varphi \rangle_{A^2}| \leq 
\|\varphi\|_{(A^1)^\ast} \|fg\|_{A^1} \leq \|\varphi\|_{(A^1)^\ast} 
\|f\|_{A^2}\|g\|_{A^2},\]
i.e. the Hankel form $H_\varphi$ is bounded on $A^2(\mathbb{D}^\infty) \times 
A^2(\mathbb{D}^\infty)$ in this case. Our first goal is to show that the 
converse does not hold. We define the weak product $A^2(\mathbb{D}^\infty) 
\odot A^2(\mathbb{D}^\infty)$ as the closure of all finite sums $f = \sum_k g_k 
h_k$, $g_k, h_k \in A^2(\mathbb{D}^\infty)$, under the norm
\[\|f\|_{A^2(\mathbb{D}^\infty)\odot A^2(\mathbb{D}^\infty)} = \inf \sum_k 
\|g_k\|_{A^2(\mathbb{D}^\infty)}\|h_k\|_{A^2(\mathbb{D}^\infty)}.\]
Here the infimum is taken over all finite representations $f = \sum_k g_k h_k$. 
Note that $\|f\|_{A^1(\mathbb{D}^\infty)} \leq 
\|f\|_{A^2(\mathbb{D}^\infty)\odot A^2(\mathbb{D}^\infty)}$.
\begin{lem}
	\label{lem:bsymb} Suppose that $\varphi$ generates a Hankel form on 
$A^2(\mathbb{D}^\infty)\times A^2(\mathbb{D}^\infty)$. Then
	\[\|H_\varphi\| = \|\varphi\|_{(A^2(\mathbb{D}^\infty)\odot 
A^2(\mathbb{D}^\infty))^\ast}.\]
	Every bounded Hankel form $H_\varphi$ extends to a bounded functional 
on $A^1(\mathbb{D}^\infty)$ if and only if there is a constant $C_\infty < 
\infty$ such that for any $f\in A^1(\DD^\infty)$,
	\[\|f\|_{A^2(\mathbb{D}^\infty) \odot A^2(\mathbb{D}^\infty)} \leq 
C_\infty \|f\|_{A^1(\mathbb{D}^\infty)}.\]
\end{lem}
\begin{proof}
	The first statement is a tautology. The weak product space 
$A^2(\mathbb{D}^\infty) \odot A^2(\mathbb{D}^\infty)$ is a Banach space, and 
therefore the second statement follows from the closed graph theorem and 
duality (see \cite{BP15,Helson06}). 
\end{proof}

Factorization and weak factorization of Hardy and Bergman spaces have a long 
history. Strong factorization for $H^1(\mathbb{D})$ was treated by Nehari 
\cite{Nehari57}, and the analogous factorization for $A^1(\mathbb{D})$ was 
given by Horowitz \cite{Horowitz77}. Every $f$ in $H^1(\mathbb{D})$ or 
$A^1(\mathbb{D})$ can be written as a single product $f = gh$, for $g,h$ in 
$H^2(\mathbb{D})$ or $A^2(\mathbb{D})$, respectively. In Nehari's theorem it is 
even possible to choose $g$ and $h$ such that 
$\|f\|_{H^1(\mathbb{D})}=\|g\|_{H^2(\mathbb{D})}\|h\|_{H^2(\mathbb{D})}$. The 
same is not possible in the factorization of $A^1(\mathbb{D})$, a simple 
observation we do not find recorded in the literature.

Factorization on the polydisc $\mathbb{D}^d$ is a much subtler matter, even 
when $1 < d < \infty$. Strong factorization is certainly not possible, but in 
\cite{FL02,LT09} it was shown that the corresponding weak factorization holds,
$$H^1(\mathbb{D^d}) = H^2(\mathbb{D}^d) \odot H^2(\mathbb{D}^d), \qquad d < 
\infty.$$ 
The Bergman space analogue was established in \cite{Constantin08},
$$A^1(\mathbb{D}^d) = A^2(\mathbb{D}^d) \odot A^2(\mathbb{D}^d), \qquad d < 
\infty.$$ 
In \cite{OCS12} it was shown that the best constant $C_d$ in the factorization,
\[\|f\|_{H^2(\mathbb{D}^d)\odot H^2(\mathbb{D}^d)} \leq C_d 
\|f\|_{H^1(\mathbb{D}^d)},\]
satisfies growth estimate $C_{d} \geq (\pi^2/8)^{d/4}$ when $d$ is an even 
integer. This immediately implies that the weak factorization 
$H^1(\mathbb{D}^\infty) = H^2(\mathbb{D}^\infty) \odot H^2(\mathbb{D}^\infty)$ 
is impossible. By tensorization, it is explained in \cite[Sec.~3]{BP15} that 
$C_{kd} \geq C_d^k$ for every positive integer $k$, a result which effortlessly 
carries over to the context of Bergman spaces. Hence we have the following.
\begin{thm}
	\label{thm:weakfac} Let $C_d$ denote the best constant in the inequality
	\[\|f\|_{A^2(\mathbb{D}^d)\odot A^2(\mathbb{D}^d)} \leq C_d 
\|f\|_{A^1(\mathbb{D}^d)},\]
	for $d=1,2,\ldots$. Then
	\[C_d \geq \left(\frac{9}{8}\right)^{d/2}.\]
	In particular, the factorization in the unit disc is not 
norm-preserving, and therefore the weak factorization
	$$A^1(\mathbb{D}^\infty) = A^2(\mathbb{D}^\infty) \odot 
A^2(\mathbb{D}^\infty)$$
	 does not hold.
\end{thm}
\begin{proof}
	In view of the discussion preceeding the theorem, it is sufficient to 
prove that $C_1 \geq 3/(2\sqrt{2})$. For every polynomial $\varphi$, we get 
from duality that
	\[C_1 \geq 
\frac{\|\varphi\|_{(A^1(\mathbb{D}))^\ast}}{\|\varphi\|_{(A^2(\mathbb{D})\odot 
A^2(\mathbb{D}))^\ast}} \geq 
\frac{\|\varphi\|_{A^2(\mathbb{D})}^2}{\|\varphi\|_{A^1(\mathbb{D})}\|\varphi\|_
{(A^2(\mathbb{D})\odot A^2(\mathbb{D}))^\ast}},\]
	where we have estimated the $(A^1(\mathbb{D}))^\ast$-norm by testing 
$\varphi$ against itself. As in Lemma~\ref{lem:bsymb}, we have that
	\[\|\varphi\|_{(A^2(\mathbb{D})\odot A^2(\mathbb{D}))^\ast} = 
\|H_\varphi\|_{A^2(\mathbb{D})\times A^2(\mathbb{D})}.\]
	We choose $\varphi(w) = \sqrt{2}\,w$. Clearly 
$\|\varphi\|_{A^2(\mathbb{D})}=1$. The matrix of $H_\varphi$ with respect to 
the standard basis of $A^2(\mathbb{D})$ is
	\[
	\begin{pmatrix}
		0 & 1 \\
		1 & 0 
	\end{pmatrix}
	,\]
	so we find that $\|H_\varphi\|_{A^2(\mathbb{D})\times 
A^2(\mathbb{D})}=1$. We are done, since
	\[\|\varphi\|_{A^1(\mathbb{D})} = 2\sqrt{2}\int_0^1 r^2\,dr = 
\frac{2\sqrt{2}}{3}. \qedhere \]
\end{proof}

It would be interesting to decide if the symbol of the Hilbert--type form 
considered in Theorem~\ref{thm:hilberttype}, which lifts to
\begin{equation} \label{eq:hilbertlift}
	\varphi(z) = \sum_{n=2}^\infty 
\frac{d(n)}{\sqrt{n}(\log{n})^2}\,z^{\kappa(n)},
\end{equation} defines a bounded linear functional on $H^1(\mathbb{D}^\infty)$. 
We are unable to settle this problem, but offer the following two observations. 
First, if $f$ is an analytic polynomial on $\mathbb{D}^\infty$ such that 
$f(0)=0$, we may write 
\[\langle f, \varphi \rangle_{A^2(\mathbb{D}^\infty)} = \int_{1/2}^\infty 
\left(\mathscr{B}^{-1}f\right)(\sigma+it)\, 
\left(\sigma-\frac{1}{2}\right)d\sigma.\] 
If we could prove the embedding 
$\|f\|_{A^1_{\operatorname{i}}(\mathbb{C}_{1/2})} \leq 
\widetilde{C}\|f\|_{\mathscr{A}^1}$, which is a stronger version of the second 
statement in Theorem~\ref{thm:carlemanhalf}, then it would follow by simple 
Carleson measure argument that \eqref{eq:hilbertlift} defines a bounded linear 
functional on $H^1(\mathbb{D}^\infty)$, through the (inverse) Bohr lift.

Our second observation is contained in the following result. 
\begin{thm} \label{thm:duality}
	Let $\varphi$ be as in \eqref{eq:hilbertlift}. Then $\varphi$ defines a 
bounded functional on $A^p(\mathbb{D}^\infty)$ for every $1<p<\infty$.
\end{thm}
\begin{proof}
	This is trivial when $p\geq2$, since $\varphi \in 
H^2(\mathbb{D}^\infty)$. Let us therefore fix $1<p<2$, and suppose that $f(z) = 
\sum_{n\geq1} a_n z^{\kappa(n)}$ is in $A^p(\mathbb{D}^\infty)$. Then it 
follows from the Cauchy--Schwarz inequality and Lemma~\ref{lem:weisslerpoly} 
with $r_j = \sqrt{p/2}$ that
	\begin{align*}
		\left|\langle f, \varphi 
\rangle_{A^2(\mathbb{D}^\infty)}\right| &= \left|\sum_{n=2}^\infty a_n 
\,\frac{1}{\sqrt{n}(\log{n})^2}\right| \\
		&\leq \left(\sum_{n=2}^\infty 
\frac{|a_n|^2}{d(n)}\,\left(\frac{p}{2}\right)^{\Omega(n)}\right)^\frac{1}{2} 
\left(\sum_{n=2}^\infty 
\left(\frac{2}{p}\right)^{\Omega(n)}\frac{d(n)}{n(\log{n})^4}\right)^\frac{1}{2}
\\
		&\leq \|f\|_{A^p(\mathbb{D}^\infty)}\left(\sum_{n=2}^\infty 
\left(\frac{2}{p}\right)^{\Omega(n)}\frac{d(n)}{n(\log{n})^4}\right)^\frac{1}{2}
	\end{align*}
	where again $\Omega(n)$ denotes the number of prime factors of $n$. We 
may conclude if we can show that 
	\[\sum_{n=2}^\infty \frac{d(n)\alpha^{\Omega(n)}}{n(\log{n})^4} < 
\infty\]
	if $1<\alpha<2$. This follows at once from Abel summation and the 
estimate
	\begin{equation} \label{eq:avgord}
		\frac{1}{x}\sum_{n\leq x} d(n)\alpha^{\Omega(n)} = C_\alpha 
(\log{x})^{2\alpha-1}+ O\left((\log{x}^{2\alpha-2})\right).
	\end{equation}
	To demonstrate \eqref{eq:avgord}, we consider the associated Dirichlet 
series, for say $\mre(s)>1$, and factor out an appropriate power of the zeta 
function
	\begin{align*}
		f_\alpha(s) &= \sum_{n=1}^\infty d(n)\alpha^{\Omega(n)}n^{-s} = 
\prod_{j=1}^\infty \left(\frac{1}{1-\alpha p_j^{-s}}\right)^2 \\
		&= \zeta^{2\alpha}(s) \prod_{j=1}^\infty 
\left(\frac{(1-p_j^{-s})^\alpha}{1-\alpha p_j^{-s}}\right)^2 =: 
\zeta^{2\alpha}(s)g_\alpha(s).
	\end{align*}
	Note that since
	\[\left(\frac{(1-p_j^{-s})^\alpha}{1-\alpha p_j^{-s}}\right)^2 = 1 + 
(\alpha-1)\alpha\,p_j^{-2s} + O(p_j^{-3s}),\]
	the Dirichlet series $g_\alpha$ is absolutely convergent for
	\[\mre(s)> \max\left(1/2,\,\log_2{\alpha}\right).\]
	A standard residue integration argument (see 
e.g.~\cite[Ch.~II.5]{Tenenbaum}) now gives \eqref{eq:avgord} with $C_\alpha = 
g_\alpha(1)/\Gamma(2\alpha)$.
	\end{proof}

Next, we investigate Hilbert--Schmidt Hankel forms \eqref{eq:mhankd}, following 
\cite{Helson06}. Recall that on the finite polydisc $\mathbb{D}^d$, $d < 
\infty$, a symbol $\varphi$ generates a Hilbert--Schmidt Hankel form on 
$H^2(\mathbb{D}^d)\times H^2(\mathbb{D}^d)$ if and only if it generates a 
Hilbert--Schmidt Hankel form on $A^2(\mathbb{D}^d)\times A^2(\mathbb{D}^d)$. On 
the infinite polydisc we have the following result. 
Theorem~\ref{thm:carlemanhalf} is its essential ingredient.

\begin{thm}
	\label{thm:HSforms} If the Hankel form generated by $\varphi$ is 
Hilbert--Schmidt on $A^2(\mathbb{D}^\infty) \times A^2(\mathbb{D}^\infty)$, 
then $\varphi$ also generates a bounded functional on $A^1(\mathbb{D}^\infty)$. 
If $\varphi$ generates a Hilbert--Schmidt form on $H^2(\mathbb{D}^\infty) 
\times H^2(\mathbb{D}^\infty)$, then it generates a Hilbert--Schmidt form on 
$A^2(\mathbb{D}^\infty) \times A^2(\mathbb{D}^\infty)$, but the converse does 
not hold. 
\end{thm}
\begin{proof}
	First, we compute the Hilbert--Schmidt norm on $A^2(\mathbb{D}^\infty) 
\times A^2(\mathbb{D}^\infty)$ of the form $H_\varphi$ generated by the symbol 
$\varphi(s) = \sum_{n\geq1} \overline{\varrho_n} z^{\kappa(n)}$. An orthonormal 
basis for $A^2(\mathbb{D}^\infty)$ is given by
	\[e_n(z) = z^{\kappa(n)} \sqrt{d(n)}.\]
	Hence, 
	\begin{align*}
		\|H_\varphi\|_{S_2(A^2(\mathbb{D}^\infty) \times 
A^2(\mathbb{D}^\infty))}^2 &= \sum_{m=1}^\infty \sum_{n=1}^\infty 
|H_\varphi(e_m e_n)|^2 = \sum_{m=1}^\infty \sum_{n=1}^\infty 
\frac{|\varrho_{mn}|^2 d(m)d(n)}{[d(mn)]^2}\\
		&=\sum_{l=1}^\infty \frac{|\varrho_l|^2}{[d(l)]^2} \sum_{mn=l} 
d(m)d(n) = \sum_{l=1}^\infty |\varrho_l|^2 \frac{d_4(l)}{[d(l)]^2}, 
	\end{align*}
	where we have made use of \eqref{eq:multconv} after recalling the 
convention that $d_2 = d$. The first statement now follows from 
Theorem~\ref{thm:carlemanhalf}, since the Cauchy--Schwarz inequality implies 
that
	\[|\langle f, \varphi \rangle_{A^2(\mathbb{D}^\infty)}| = 
\left|\sum_{n=1}^\infty \frac{a_n \varrho_n}{d(n)} \right| \leq 
\left(\sum_{n=1}^\infty 
\frac{|a_n|^2}{d_4(n)}\right)^\frac{1}{2}\left(\sum_{n=1}^\infty |\varrho_n|^2 
\,\frac{d_4(n)}{[d(n)]^2}\right)^\frac{1}{2}.\]
	Similarly we have that
	\[\|H_\varphi\|_{S_2(H^2(\mathbb{D}^\infty) \times 
H^2(\mathbb{D}^\infty))}^2 = \sum_{n=1}^\infty |\varrho_n|^2 d(n).\]
	Note that when $n$ is a prime power $n = p_j^k$ we have that
	\[d_4(n) = \frac{(k+1)(k+2)(k+3)}{6} \leq (k+1)^3 = [d(n)]^3.\]
	Since both $d_4(n)$ and $d(n)$ are multiplicative functions, it follows 
that $d_4(n) \leq [d(n)]^3$ for every $n$. Hence the second statement is proved.
	
	To see that the converse of the second statement does not hold, 
consider the set $\mathscr{N} = \{n_1=2,\, n_2 = 3\cdot 5,\, n_3 = 7\cdot 11 
\cdot 13,\, \ldots\,\}$ and define $\varphi(s) = \sum_{n \in \mathscr{N}} 
\overline{\varrho_n} z^{\kappa(n)}$. Then we have that
	\begin{align*}
		\|H_\varphi\|_{S_2(A^2(\mathbb{D}^\infty) \times 
A^2(\mathbb{D}^\infty))}^2 &= \sum_{j=1}^\infty |\varrho_{n_j}|^2, \\
		\|H_\varphi\|_{S_2(H^2(\mathbb{D}^\infty) \times 
H^2(\mathbb{D}^\infty))}^2 &= \sum_{j=1}^\infty |\varrho_{n_j}|^2 2^j. \qedhere
	\end{align*}
\end{proof}

The final part of this section is devoted to showing that every Hankel form of 
the type \eqref{eq:mhankd} naturally corresponds to a Hankel form of the type 
\eqref{eq:mhank} with the same singular numbers. Let $D$ denote the diagonal 
operator in two variables, $Df(w)=f(w,w)$, for which we have the following 
observation. 
\begin{lem}
	\label{lem:Dnorm} The operator $D$ is a contraction from 
$H^2(\mathbb{D}^2)$ to $A^2(\mathbb{D})$. 
\end{lem}
\begin{proof}
	This is proven in \cite{Rudin69}, but in an abstract formulation it 
dates back at least to Aronzajn \cite{Arons50}. The proof of our particular 
case is very easy and we include it here. Consider
	\[f(z_1,z_2) = \sum_{j=0}^\infty \sum_{k=0}^\infty a_{j,k} z_1^j z_2^k\]
	and use the Cauchy--Schwarz inequality to conclude that
	\[\|Df\|_{A^2(\mathbb{D})}^2 = \sum_{l=0}^\infty 
\frac{1}{l+1}\left|\sum_{j+k=l} a_{j,k}\right|^2 \leq \sum_{l=0}^\infty 
\sum_{j+k=l} |a_{j,k}|^2 = \|f\|_{H^2(\mathbb{D}^2)}^2.\qedhere\]
\end{proof}
The diagonal operator $D$ may be written as an integral operator using the 
reproducing kernel of $H^2(\mathbb{D}^2)$,
\[Df(w) = \int_{\mathbb{T}^2} 
f(z_1,z_2)\,\frac{1}{1-w\overline{z_1}}\frac{1}{1-w\overline{z_2}}\,
dm_1(z_1)dm_1(z_2).\]
Hence its adjoint operator $E\colon A^2(\mathbb{D}) \to H^2(\mathbb{D}^2)$ is 
given by
\[Eg(z_1,z_2) = 
\int_{\mathbb{D}}g(w)\,\frac{1}{1-z_1\overline{w}}\frac{1}{1-z_2\overline{w}}\,
dA(w).\]
If $f$ and $g$ are in $A^2(\mathbb{D})$, then
\[\langle Ef, E g \rangle_{H^2(\mathbb{D}^2)} = \langle f, g 
\rangle_{A^2(\mathbb{D})},\]
that is, $E$ is an isometry. Clearly, the composition $DE$ is the identity 
operator on $A^2(\mathbb{D})$. Hence we have identified $A^2(\mathbb{D})$ with 
the subspace $X = E A^2(\mathbb{D})$ of $H^2(\mathbb{D}^2)$ (although perhaps 
it would be more appropriate to think of it as the factor space induced by the 
map $D$). The projection $P\colon H^2(\mathbb{D}^2)\to X$ is given by $P=ED$. 
Note that $P$ averages the coefficients of monomials of same degree. Precisely, 
if $f(z) = \sum_{j,k\geq0} a_{j,k}z_1^j z_2^k$, then
\[Pf(z_1,z_2) = \sum_{j=0}^\infty \sum_{k=0}^\infty A_{j+k} z_1^j z_2^k, \qquad 
\text{ where } \qquad A_l = \frac{1}{l+1}\sum_{j+k=l} a_{j,k}.\]
Clearly, $D(fg) = D(f)D(g)$, but $E$ does not have this property. For example, 
if $g(w)=w$, then
\[Eg(z_1,z_2) = \frac{z_1+z_2}{2} \qquad \text{and} \qquad E(g^2)(z_1,z_2) = 
\frac{z_1^2 + z_1z_2 + z_2^2}{3},\]
so that $E(g)E(g) \neq E(g^2)$. 

Let us now turn to the relationship between the operator $E$ and Hankel forms. 
To fix the notation, let $Y$ be a Hilbert space with an orthonormal basis 
$\{e_j\}_{j\geq1}$. For a bilinear form $H \colon Y \times Y \to \mathbb{C}$, 
let $s_n(H)$ denote its $n$th singular value, i.e. $$s_n(H) = \inf\{\|H - 
K\|_{Y \times Y} \, : \, \rank K \leq n \},$$ where the rank of a bilinear form 
$K : Y \times Y \to \mathbb{C}$ is given by $$\rank K = \codim \ker K = \codim 
\{f \in Y \, : \, K(f,g) = 0 \textrm{ for all } g \in Y\}. $$ Of course, 
$s_n(H)$ is the same as the $n$th singular value of the operator $\{H(e_j, 
e_k)\}_{j,k\geq1} : \ell^2 \to \ell^2$. The $p$-Schatten norm of $H$, $0 < p < 
\infty$, is given by $$\|H\|_{S_p(Y \times Y)}^p = \sum_{n=0}^\infty 
|s_n(H)|^p.$$ When $p = 2$ we obtain the Hilbert-Schmidt norm, which can also be 
computed as the square sum of the coefficients,
$$\|H\|_{S_2(Y \times Y)}^2 = \sum_{n=0}^\infty |s_n(H)|^2 = 
\sum_{j=1}^\infty\sum_{k=1}^\infty |H(e_j, e_k)|^2.$$ 
We have the following result. 
\begin{lem}
	\label{lem:DEhankel} Suppose that $\varphi \in A^2(\mathbb{D})$. Then 
$$s_n(H_{\varphi}) = s_n(H_{E\varphi}), \qquad n \geq 0.$$ In particular, for 
$0 < p < \infty$ we have 
	\begin{align*}
		\|H_{\varphi}\|_{A^2(\mathbb{D}) \times A^2(\mathbb{D})} &= 
\|H_{E\varphi}\|_{H^2(\mathbb{D}^2)\times H^2(\mathbb{D}^2)}, \\
		\|H_{\varphi}\|_{S_p(A^2(\mathbb{D}) \times A^2(\mathbb{D}))} 
&= \|H_{E\varphi}\|_{S_p(H^2(\mathbb{D}^2)\times H^2(\mathbb{D}^2))}. 
	\end{align*}
\end{lem}
\begin{proof}
	Let $J : X \times X \to \mathbb{C}$ be the restriction of 
$H_{E\varphi}$ to $X = EA^2(\mathbb{D})$, 
	$$J(f, g) = \langle fg, E\varphi \rangle_{H^2(\mathbb{D}^2)}, \qquad 
f,g \in X.$$
	 For $f, g \in H^2(\mathbb{D}^2)$ we have the identity 
	\begin{equation}
		\label{eq:Edual} \langle fg, E\varphi 
\rangle_{H^2(\mathbb{D}^2)} = \langle D(fg), \varphi \rangle_{A^2(\mathbb{D})} 
= \langle Df Dg, \varphi \rangle_{A^2(\mathbb{D})}. 
	\end{equation}
	Since $D : X \to A^2(\mathbb{D})$ is unitary, this
	implies that $J$ is unitarily equivalent to $H_\varphi : 
A^2(\mathbb{D}) \times A^2(\mathbb{D}) \to \mathbb{C}$. If $K : 
H^2(\mathbb{D}^2) \times H^2(\mathbb{D}^2) \to \mathbb{C}$ is a rank-$n$ form, 
then its restriction to $X$, $K' : X \times X \to \mathbb{C}$, has smaller rank, 
$\rank K' \leq n$. Since $$\|H_{E\varphi} - K\|_{H^2(\mathbb{D}^2)\times 
H^2(\mathbb{D}^2)} \geq \|J - K'\|_{X \times X}$$ it follows that 
	\[s_n(H_{E\varphi}) \geq s_n(J) = 
s_n(H_\varphi), \qquad n \geq 0.\]
	
	Conversely, if the form $K : A^2(\mathbb{D}) \times A^2(\mathbb{D}) \to 
\mathbb{C}$ has rank $n$, then clearly $K' : H^2(\mathbb{D}^2) \times 
H^2(\mathbb{D}^2) \to \mathbb{C}$ has smaller rank, where $K'(f, g) = K(Df, 
Dg)$, for $f,g \in H^2(\mathbb{D}^2)$. However, it follows from 
\eqref{eq:Edual} and Lemma~\ref{lem:Dnorm} that $$\|H_{\varphi} - K\| = 
\|H_{E\varphi} - K'\|,$$ proving that also $s_n(H_\varphi) \geq 
s_n(H_{E\varphi})$. 
\end{proof}

Consider $A^2(\mathbb{D}^\infty)$ as a function space over the variables $z = 
(z_1,z_2,\ldots)$ and $H^2(\mathbb{D}^\infty)$ as a function space over $\xi = 
(\xi_1,\xi_2,\ldots)$. Define the extension map $\mathscr{E}$ from 
$A^2(\mathbb{D}^\infty)$ to $H^2(\mathbb{D}^\infty)$ by its integral kernel,
\[K_\xi(z) = \prod_{j=1}^\infty 
\frac{1}{1-\xi_{2j-1}\overline{z_j}}\,\frac{1}{1-\xi_{2j}\overline{z_j}}, 
\qquad z, \xi \in \mathbb{D}^\infty \cap \ell^2,\]
so that
\[\mathscr{E}f(\xi) = \int_{\mathbb{D}^\infty} f(z)K_\xi(z)\,d\mathbf{m}(z).\]
By tensorization of Lemma~\ref{lem:DEhankel} (the required technical details 
may be found in \cite[Lem.~2]{BP15}), we obtain the following. 
\begin{thm}
	\label{thm:extension} The map $\mathscr{E}$ has the following 
properties. 
	\begin{enumerate}
		\item[(a)] $\mathscr{E}$ defines an isometric isomorphism from 
the Bergman space $A^2(\mathbb{D}^\infty)$ to a subspace of the Hardy space 
$H^2(\mathbb{D}^\infty)$. 
		\item[(b)] For $\varphi \in A^2(\mathbb{D}^\infty)$, let 
$H_\varphi : A^2(\mathbb{D}^\infty)\times A^2(\mathbb{D}^\infty) \to 
\mathbb{C}$ be the Hankel form generated by $\varphi$, and let 
$H_{\mathscr{E}\varphi} : H^2(\mathbb{D}^\infty) \times H^2(\mathbb{D}^\infty) 
\to \mathbb{C}$ be the Hankel form generated by $\mathscr{E}\varphi$. Then, for 
every $n \geq 0$, we have that $$s_n(H_\varphi) = s_n(H_{\mathscr{E}\varphi}).$$ 
In particular, $H_\varphi$ is bounded ($p$-Schatten, $0 < p < \infty$) if and 
only if $H_{\mathscr{E}\varphi}$ is bounded ($p$-Schatten), with equality of 
the norms. 
	\end{enumerate}
\end{thm}
\begin{remark}
	In \cite{OCS12}, the symbol $\psi(z) = (z_1 + z_2)/2$ is used to show 
that the weak factorization
	$H^1(\mathbb{D}^\infty) = H^2(\mathbb{D}^\infty) \odot 
H^2(\mathbb{D}^\infty)$
	cannot hold. In Theorem~\ref{thm:weakfac} the symbol $\varphi(w) = w$ 
is used to demonstrate the corresponding fact for the Bergman spaces. In fact 
the two examples considered are the same, because $E\varphi = \psi$. 
\end{remark}
\section{Carleson measures on the infinite polydisc}\label{sec:carleson} We end 
this paper by producing two infinite dimensional counter-examples to well-known 
finite dimensional results for Carleson measures for the Hardy spaces 
$H^p(\mathbb{D}^d)$. Let $\mu$ be a finite positive measure on $\mathbb{D}^d$ 
(where possibly $d=\infty$), i.e. a finite positive Borel measure on 
$\overline{\mathbb{D}}^d$ such that $\mu(\overline{\mathbb{D}}^d \setminus 
\mathbb{D}^d) = 0$. As usual, measures on the compact space 
$\overline{\mathbb{D}}^d$ correspond to linear functionals on the space of 
continuous functions $C(\overline{\mathbb{D}}^d)$. We say that $\mu$ is a 
$H^p$-Carleson measure if there exists a constant $C=C(\mu_d,p)<\infty$ such 
that
\[\int_{\mathbb{D}^d}|f(z)|^p \, d\mu_d(z) \leq C\|f\|_{H^p(\mathbb{D}^d)}^p\]
for every analytic polynomial $f$. We say that $\mu$ is a $L^p$-Carleson 
measure if there exists a constant $C=C(\mu_d,p)<\infty$ such that
\[\int_{\mathbb{D}^d}|\mathscr{P}f(z)|^p \, d\mu_d(z) \leq 
C\|f\|_{L^p(\mathbb{T}^d)}^p\]
for every trigonometric polynomial $f$. Here $\mathscr{P}f$ is the Poisson 
extension of $f$, defined for $f \in L^p(\mathbb{T}^d)$ by
\[\mathscr{P}f(w) = \int_{\mathbb{T}^d}f(z)\mathscr{P}_w(z)\,d\mathbf{m}_1(z), 
\qquad \mathscr{P}_w(z) = \prod_{j=1}^d 
\frac{1-|w_j|^2}{|1-\overline{z_j}w_j|^2}.\]
This is always well-defined as long as we restrict ourselves to  
$L^2(\mathbb{T}^d)$-functions $f$ only dependent on a finite number of 
variables, since we may then suppose that $w$ is finitely supported.

The study of Carleson measures on the infinite polydisc is an important part of 
the theory of $H^p$ spaces. For instance, the local embedding problem discussed 
in \cite[Sec.~3]{SS09} can be formulated in terms of Carleson measures. Let 
$\mathscr{B}^{-1}$ denote the inverse Bohr lift, so that
\[\big(\mathscr{B}^{-1}f\big)(s) = 
f\big(2^{-s},\,3^{-s},\,5^{-s},\,\ldots\,p_j^{-s},\,\ldots\big).\]
For $0<p<\infty$, is it true that the measure $\mu_\infty$ defined on 
$\mathbb{D}^\infty$ by 
\[\int f(z)d\mu_\infty(z)=\int_0^1 \big(\mathscr{B}^{-1}f\big)(1/2+it)\,dt, 
\qquad f \in C(\overline{\mathbb{D}}^d),\]
is a $H^p$-Carleson measure? A positive answer is only known for even 
integers. Additionally, the boundedness of positive definite Hankel forms 
\eqref{eq:mhank} can be formulated in terms of Carleson measures on 
$\mathbb{D}^\infty$ \cite{PP16}, and the same is true for the Volterra operators 
studied in \cite{BPS16}.

From \cite{Ch79}, it is known that a measure $\mu$ on $\DD^d$, for $d<\infty$, 
is a $H^p$-Carleson measure for one $0<p<\infty$ if and only if it is a 
Carleson measure for every $0<p<\infty$. We will now construct a 
counter-example to this statement when $d=\infty$. We recall that the diagonal 
restriction operator $Df(w) = f(w,w)$ induces a bounded map from 
$H^p(\mathbb{D}^2)$ to $A^p(\mathbb{D})$ for every $0<p<\infty$ 
(see~\cite{DS75}), and offer the following clarification in the case $0<p<2$.
\begin{lem}
	\label{lem:Dpnorm} The diagonal operator $D$ is not contractive from 
$H^p(\mathbb{D}^2)$ to $A^p(\mathbb{D})$ when $0<p<2$. 
\end{lem}
\begin{proof}
	Let $0<p<2$ and consider $f(z_1,z_2) = (z_1+z_2)/2$. Clearly
	\[\|Df\|_{A^p(\mathbb{D})}^p = \int_{\mathbb{D}} |f(w,w)|^p \, dm(w) = 
\frac{2}{2+p},\]
	so it is enough to verify that $\|f\|_{H^p(\mathbb{D}^2)}^p < 2/(2+p)$. 
We factor out $z_2$ and compute using various identities for the Beta and Gamma 
functions, obtaining that
	\begin{align*}
		\|f\|_{H^p(\mathbb{D}^2)}^p &= \frac{1}{2\pi} \int_0^{2\pi} 
\left|\frac{1+e^{i\theta}}{2}\right|^p \, d\theta = \frac{1}{2\pi} 
\int_0^{2\pi} \left|\cos{\frac{\theta}{2}}\right|^p \,d\theta \\
		&= \frac{1}{\pi} \int_0^{\pi} 
\left(\cos{\frac{\theta}{2}}\right)^p \, d\theta = \frac{2}{\pi} \int_0^1 
\frac{t^p}{\sqrt{1-t^2}}\,dt \\
		&= \frac{1}{\pi} \int_0^1 t^{(p-1)/2}(1-t)^{-1/2}\,dt = 
\frac{\operatorname{B}((p+1)/2),1/2)}{\pi} \\
		&= \frac{\Gamma(p/2+1/2)\Gamma(1/2)}{\pi\Gamma(p/2+1)} = 
\frac{\Gamma(p/2+1/2)}{\Gamma(1/2)(p/2)\Gamma(p/2)} = 
\frac{2}{p\operatorname{B}(p/2,1/2)}. 
	\end{align*}
	To conclude we make use of the identity
	\[\operatorname{B}(x,y) = \sum_{n=0}^\infty \binom{n-y}{n} 
\frac{1}{x+n}, \qquad x,y>0.\]
	The binomial coefficient is positive for every $n$ when $y=1/2$, so if 
$0<p<2$ we have that
	\[\operatorname{B}(p/2,1/2) > \frac{1}{p/2} + \operatorname{B}(1,1/2) - 
\frac{1}{1} = \frac{2}{p} + 1. \qedhere\]
\end{proof}
\begin{remark}
	Lemma~\ref{lem:Dnorm} implies that $D$ is a contraction from 
$H^p(\mathbb{D}^2)$ to $A^p(\mathbb{D})$ if $p$ is an even integer. It would be 
interesting to know if $D$ is a contraction for every $p\geq2$.
\end{remark} 
Tensorization of Lemma~\ref{lem:Dnorm} and Lemma~\ref{lem:Dpnorm} yields the 
following result. 
\begin{thm}
	\label{thm:carleson} Let $\mu_\infty$ be the measure defined for $f$ in 
$C(\overline{\mathbb{D}}^\infty)$ by
	\begin{equation} \label{eq:diagcarleson}
		\int_{\mathbb{D}^\infty} f(z_1,z_2,z_3,z_4,\ldots)\, d\mu_\infty(z) = 
		\int_{\mathbb{D}^\infty} f(z_1,z_1,z_3,z_3,\ldots) \,d\mathbf{m}(z),
	\end{equation}
	where $\mathbf{m}$ denotes the infinite product of the unweighted 
normalized Lebesgue measure on $\mathbb{D}$. The measure $\mu_\infty$ is a 
$H^p$-Carleson measure on $\mathbb{D}^\infty$ if $p$ is an even integer, but 
not when $0<p<2$. 
\end{thm}
Theorem~\ref{thm:carleson} invites the following question.
\begin{question}
	If $\mu$ defines a $H^p$-Carleson measure on $\DD^\infty$ for some 
$0<p<\infty$, does it also define a $H^q$-Carleson measure for every 
$p<q<\infty$? 
\end{question}

In \cite{Ch79}, it is also proven that $L^p$-Carleson and $H^p$-Carleson 
measures coincide on $\DD^d$, when $d<\infty$. Again, this is no longer true on 
$\DD^\infty$, as our next two examples will demonstrate. 

To obtain the first counter-example, we verify that the 
measure \eqref{eq:diagcarleson} of Theorem~\ref{thm:carleson} does not define a 
$L^2$-Carleson measure on $\mathbb{D}^\infty$ by replacing Lemma~\ref{lem:Dpnorm} 
with the following result.

\begin{lem} \label{lem:DL2norm}
	The operator $D \circ \mathscr{P}$ is not a contraction from $L^2(\mathbb{T}^2)$ to $L^2(\mathbb{D},m)$.
\end{lem}
\begin{proof}
	Consider
	\[f(e^{i\theta_1},e^{i\theta_2}) = \frac{1}{\sqrt{3}}\left(e^{i\theta_1}+e^{i\theta_2} + e^{i2\theta_1}e^{-i\theta_2}\right),\]
	for which clearly $\|f\|_{L^2(\mathbb{T}^2)}=1$. Furthermore, we find that 
	\[\mathscr{P}f(re^{i\theta},re^{i\theta}) = \frac{e^{i\theta}}{\sqrt{3}}(2r+r^3),\]
	so it follows that
	\[\int_{\mathbb{D}} |\mathscr{P}f(z,z)|^2\,dm(z) = \frac{2}{3} \int_0^1 \left(2r+r^3\right)^2\,rdr = \frac{43}{36}>1. \qedhere\]
\end{proof}

Our second counter-example is obtained through the connection with Dirichlet series. 
In preparation, let us recall a few properties of $L^2(\mathbb{T}^\infty)$. Let 
$\mathbb{Q}_+$ denote the set of positive rational numbers. Each $q\in\mathbb{Q}_+$ 
has a finite expansion of the form
\[q = \prod_{j=1}^\infty p_j^{\kappa_j},\]
where $\kappa_j \in \mathbb{Z}$. Hence $\mathbb{Q}_+$ can be identified with 
the set of all finite multi-indices. As in \eqref{eq:multexp}, every function $f 
\in L^2(\mathbb{T}^\infty)$ has an expansion
\[f(z) = \sum_{q \in \mathbb{Q}_+} a_q z^{\kappa(q)}, \qquad \qquad 
\|f\|_{L^2(\mathbb{T}^\infty)}^2 = \sum_{q \in \mathbb{Q}_+} |a_q|^2.\]
Note that if $f \in L^2(\mathbb{T}^{d'})$ for some $d' < \infty$ and 
$s=\sigma+it$, then
\[\big(\mathscr{B}^{-1}\mathscr{P}f\big)(s) = \sum_{q \in \mathbb{Q}_+} a_q 
(q_+)^{-\sigma} q^{-it},\qquad\text{where}\qquad q_+ = \prod_{j=1}^\infty 
p_j^{|\kappa_j|}.\]
As our final preliminary, let $\omega(n)$ denote the number of distinct prime 
factors of $n$. It is well-known that if $\mre(s)>1$, then
\[\frac{[\zeta(s)]^2}{\zeta(2s)} = \prod_{j=1}^\infty 
\frac{1+p_j^{-s}}{1-p_j^{-s}} = \sum_{n=1}^\infty 2^{\omega(n)}n^{-s}.\] 
\begin{thm}
	\label{thm:carleson2} Let $\mu_\infty$ be the measure defined for $f$ 
in $C(\overline{\mathbb{D}}^\infty)$ by
	\[\int_{\mathbb D^{\infty}}f(z)\,d\mu_\infty(z)=\int_0^1 
\big(\mathscr{B}^{-1}f\big)(1/2+\sigma)\,d\sigma.\]
	Then $\mu_\infty$ is a $H^2$-Carleson measure but not a 
$L^2$-Carleson measure. 
\end{thm}
\begin{proof}
	It is well-known that $\mu_\infty$ is a $H^2$-Carleson measure 
\cite{BPSSV,OS12}. Let us therefore prove that $\mu_\infty$ is not a 
$L^2$-Carleson measure. Fix $\varepsilon>0$ and define $w \in 
\mathbb{D}^\infty \cap \ell^2$ by $w_j = p_j^{-1/2-\varepsilon}$. We will 
consider the kernel of the $d$-dimensional Poisson transform, 
	$$f_d(z) = \prod_{j=1}^d \frac{1-|w_j|^2}{|1-\overline{z_j}w_j|^2}.$$ 
	First observe that
	\[\lim_{d\to\infty}\|f_d\|_{L^2(\mathbb{T}^\infty)}^2 = 
\prod_{j=1}^\infty \frac{1-p_j^{-2-4\varepsilon}}{(1-p_j^{-1-2\varepsilon})^2} 
=  \frac{[\zeta(1+2\varepsilon)]^2}{\zeta(2+4\varepsilon)} \simeq 
\varepsilon^{-2}.\]
	Next, we have that
	\[\lim_{d \to \infty} \big(\mathscr{B}^{-1}\mathscr{P} f_d 
\big)(1/2+\sigma) = \sum_{q \in \mathbb{Q}_+} q_+^{-1-\varepsilon-\sigma},\]
	uniformly convergent in $\sigma \in [0, 1]$. Note that
	\[\sum_{q \in \mathbb{Q}_+} q_+^{-1-\varepsilon-\sigma} = 
\sum_{n=1}^\infty 2^{\omega(n)}n^{-1-\varepsilon-\sigma} \simeq 
(\sigma+\varepsilon)^{-2},\]
	since there are $2^{\omega(n)}$ rational numbers $q \in \mathbb{Q}_+$ 
such that $q_+ = n$. This concludes the argument, since
	\[\int_0^1 \frac{d\sigma}{(\sigma+\varepsilon)^4} \simeq 
\varepsilon^{-3}. \qedhere\]
\end{proof}

\bibliographystyle{amsplain} 
\bibliography{polyberg}

\end{document}